\documentclass[review]{elsarticle}
\usepackage[top=2.5cm, bottom=2.5cm, left=2.5cm, right=2cm]{geometry}
\usepackage{amssymb,latexsym,amscd,amsmath,amsfonts,enumerate,supertabular,amsthm}
\usepackage{graphicx}
\usepackage{color}
\usepackage{tabularx}
\usepackage{subfigure}
\usepackage{enumerate}
\usepackage{epstopdf}
\usepackage{natbib,float}
\usepackage{lscape}
\usepackage{hyperref}
\newtheorem{theorem}{Theorem}
\newtheorem{lemma}{Lemma}

\newtheorem{proposition}{Proposition}
\newtheorem{corollary}{Corollary}
\newtheorem{remark}{Remark}
\newtheorem{example}{Example}

\journal{Mathematical Methods in the Applied Sciences}

\begin{document}

\begin{frontmatter}

\title{Dynamical analysis of a generalized hepatitis B epidemic model and\\ its dynamically consistent discrete model}


\author[Manh Tuan Hoang]{Manh Tuan Hoang\corref{mycorrespondingauthor}}
\cortext[mycorrespondingauthor]{Corresponding author}
\ead{tuanhm14@fe.edu.vn; hmtuan01121990@gmail.com}
\address[mymainaddress]{Department of Mathematics, FPT University, Hoa Lac Hi-Tech Park, Km29 Thang Long Blvd, Hanoi, Viet Nam}
\begin{abstract}
The aim of this work is to study qualitative dynamical properties of a generalized hepatitis B epidemic model and its dynamically consistent discrete model. Positivity, boundedness, the basic reproduction number and asymptotic stability properties of the model are analyzed rigorously. By the Lyapunov stability theory and the Poincare-Bendixson theorem in combination with the Bendixson-Dulac criterion, we show that a disease-free equilibrium point is globally asymptotically stable if the basic reproduction number $\mathcal{R}_0 \leq 1$ and a disease-endemic equilibrium point is globally asymptotically stable whenever $\mathcal{R}_0 > 1$. Next, we apply the Mickens’ methodology to propose a dynamically consistent nonstandard finite difference (NSFD) scheme for the continuous model. By rigorously mathematical analyses, it is proved that the constructed NSFD scheme preserves essential mathematical features of the continuous model for all finite step sizes. Finally, numerical experiments are conducted to illustrate the theoretical findings and to demonstrate advantages of the NSFD scheme over standard ones. The obtained results in this work not only improve but also generalize some existing recognized works.
\end{abstract}
\end{frontmatter}
\section{Introduction}\label{sec1}
Mathematical modeling and analysis of infectious diseases has been a central research field in infectious disease epidemiology with many useful applications not only in theory but also in practice \cite{Allen, Bonhoeffer, Brauer, Brauer1, Grassly, Hethcote, Kermack, Li, Martcheva, Nowak}. The well-known SIR model, proposed by Kermack and McKendrick in 1927 \cite{Kermack}, can be considered as one of the first epidemic models and is usually used to introduce epidemic modeling. Over the past of several decades, for modeling purpose, a great number of mathematical models have been studied to predict and discover transmission mechanisms of infectious diseases. As an important consequence, strategies and measures for controlling and preventing infectious diseases can be provided.\par
It is well-known that hepatitis B is a dangerous infectious disease that attacks the liver and can cause  acute and chronic diseases. Nowadays, hepatitis B has been a major global health problem. Hence, effective strategies and measures for preventing and controlling hepatitis B virus (HBV) are always needed. For this reason, a large number of mathematical models have been proposed and studied by many mathematicians and epidemiologists (see, for instance, \cite{Ahmad, Cardoso, Cardoso1, Danane, Din, Din1, Gao, Khan, Khan1, Khan2, Manna, Shah}). These mathematical models can help us predict and discover transmission mechanisms and characteristics of the HBV. Based on this, strategies and measures for projecting public health and for preventing the hepatitis B can be suggested.\par
In  \cite{Khan}, Khan et al. proposed a hepatitis B epidemic model based on HBV transmission characteristics. The model is described by a system of nonlinear differential equations
\begin{equation}\label{eq:1}
\begin{split}
&\dfrac{dS(t)}{dt} = \Lambda - \dfrac{\alpha S(t)I(t)}{1 + \gamma I(t)} - \big(\mu_0 + \nu\big)S(t),\\
&\dfrac{dI(t)}{dt} = \dfrac{\alpha S(t)I(t)}{1 + \gamma I(t)} - \big(\mu_0 + \mu_1 + \beta \big)I(t),\\
&\dfrac{dR(t)}{dt} = \beta I(t) + \nu S(t) - \mu_0 R(t),
\end{split}
\end{equation}
subject to initial data
\begin{equation*}
S(0) > 0, \quad I(0) \geq 0, \quad R(0) > 0.
\end{equation*}
In this model:
\begin{itemize}
\item the entire population, denoted by $N$, is divided into $3$ classes: susceptible ($S$), infected ($I$) and recovered or removed ($R$) classes;
\item $\Lambda$ is the birth rate;
\item $\lambda$ is the transmission rate;
\item $\mu_0$ and $\mu_1$ denote the natural and disease induced death rates, respectively;
\item $\beta$ is the recovery rate;
\item $\nu$ and $\gamma$ stand for the vaccination and saturation rates, respectively;
\end{itemize}
We refer the readers to  \cite{Khan} for more details of the model \eqref{eq:1} and its qualitative dynamical properties.  In \cite{Hoang1}, A study on the global stability and numerical solutions of the model \eqref{eq:1} was performed. Besides, Hoang and Egbelowo in \cite{Hoang} proposed and analyzed a generalized version of the model \eqref{eq:1} using the Caputo fractional derivative.\par
In the model \eqref{eq:1}, let us denote by
\begin{equation}\label{eq:2}
f(S,I) = \dfrac{\alpha S(t)I(t)}{1 + \gamma I(t)}.
\end{equation}
It is easy to verify that this function possesses the following properties:\\
(H1): $f: \mathbb{R}_+^2 \to \mathbb{R}_+$ is a differentiable function, $f(S, 0) = f(0, I) = 0$ for all $S, I \geq 0$ and $f(S, I) > 0$ for all $S, I > 0$;\\
(H2): there exists $\eta > 0$ such that $f(S, I) \leq \eta S$ for all $S, I \geq 0$;\\
(H3): $\dfrac{\partial f(S, I)}{\partial S} > 0$ and bounded for all $S \geq 0$ and $I > 0$;\\
(H4): $\dfrac{\partial f(S, I)}{\partial I} \geq 0$ for all $S, I \geq 0$;\\
(H5): $I\dfrac{\partial f(S,I)}{\partial I} - f(S, I) \leq 0$ for all $S, I \geq 0$.\\
The conditions (H1)-(H5) are biologically motivated and can be found in previous works \cite{Karaji, McCluskey, Tian}.  In a recent work \cite{Karaji}, Karaji and Nyamoradi considered the model \eqref{eq:1} in the context of the Caputo fractional-order derivative and general incidence functions. It should be emphasized that functions $f$ satisfying the conditions (H1)-(H5) include many famous incidence functions, for examples
\begin{equation*}
f(S,I) = \beta SI, \quad f(S, I) = \dfrac{\beta SI}{1 + bI}, \quad f(S, I) = \dfrac{\beta SI}{1 + aS + bI}, \quad f(S, I) = \dfrac{\beta SI}{1 + aS + bI + cSI}.
\end{equation*}
Motivated and inspired by the recognized works \cite{Karaji, McCluskey, Tian} as well as the importance of mathematical models of the HBV, in this work we consider a generalized version of the model \eqref{eq:1} by replacing the function $f$ defined by \eqref{eq:2} by general ones $f$ satisfying the conditions (H1)-(H5). More precisely, we investigate the following model
\begin{equation}\label{eq:3}
\begin{split}
&\dfrac{dS(t)}{dt} = \Lambda - f(S(t), I(t)) - \big(\mu_0 + \nu\big)S(t),\\
&\dfrac{dI(t)}{dt} = f(S(t), I(t)) - \big(\mu_0 + \mu_1 + \beta \big)I(t),\\
&\dfrac{dR(t)}{dt} = \beta I(t) + \nu S(t) - \mu_0 R(t),
\end{split}
\end{equation}
where $f(S, I)$ is any incidence function satisfying the conditions (H1)-(H5). Our main objective is to study qualitative dynamical properties and reliable numerical solutions for the model \eqref{eq:3}.\par
In the first part of this work, we analyze positivity, boundedness, the basic reproduction number, equilibria (disease-free and disease-endemic equilibrium points) and asymptotic stability properties of the model \eqref{eq:3}. By the Lyapunov stability theory and the Poincare-Bendixson theorem in combination with the Bendixson-Dulac criterion, we prove that the disease-free equilibrium (DFE) point is globally asymptotically stable if the basic reproduction number $\mathcal{R}_0 \leq 1$; and the disease-endemic equilibrium (DEE) point exists and is globally asymptotically stable whenever $\mathcal{R}_0 > 1$. Consequently, qualitative dynamical properties of the model \eqref{eq:3} are fully determined.\par
In the second part, we apply the Mickens’ methodology \cite{Mickens1, Mickens2, Mickens3, Mickens4, Mickens5} to propose a dynamically consistent nonstandard finite difference (NSFD) scheme for the model \eqref{eq:3}. By rigorously mathematical analyses, we show that the proposed NSFD scheme preserves essential mathematical features of the continuous model for all finite step sizes. In other words, the NSFD scheme behaves similarly to the continuous model regardless of chosen step sizes. It is worth noting that one of the prominent advantages of NSFD schemes is that they are able to correctly maintain important mathematical features of solutions of differential equations models (positivity, boundedness, monotonicity, stability, periodicity, physical properties, etc.) for any finite step size. Hence, NSFD schemes are effective to simulate dynamics of dynamical differential models over long time periods. Because of this, nowadays NSFD schemes have been recognized as one of the effective approaches to solve differential equation equations arising in theory and practice \cite{Adekanye, Arenas, Calatayud, Cresson, Din2, Garba, Gupta, Kojouharov, Mickens6, Patidar1, Patidar2, Wood1, Wood2}. In recent works \cite{Dang1, Dang2, Dang3, Dang4, Dang5, Dang6, Hoangnew, Hoang1}, we have successfully developed the Mickens's methodology to construct NSFD schemes for some mathematical models arising in biology and epidemiology.\par 
In the third part, we conduct numerical experiments to illustrate the theoretical findings and to demonstrate advantages of the NSFD scheme over standard ones. The numerical results show that the Euler and second-order Runge-Kutta (RK2) schemes can generate numerical approximations which destroy not only the positivity but also the asymptotic stability properties of the continuous model for some given step sizes; meanwhile, the NSFD scheme preserve these properties for the same step sizes.\par
It is worth noting that the mathematical analyses in \cite{Khan} failed to conclude the global asymptotic stability (GAS) of the DEE point of the model \eqref{eq:1}. However, we obtain the GAS of the model \eqref{eq:1} thanks to the global stability analysis for the model \eqref{eq:3} (Section \ref{sec2}). This provides an important improvement for the results in \cite{Khan}. In a recent work \cite{Suryanto}, Suryanto and Darti constructed an NSFD scheme for the model \eqref{eq:1} but the convergence and error bounds for this NSFD scheme were not established. Hence, our constructed NSFD scheme not only generalizes but also provides the convergence and error estimates for the Suryanto and Darti's scheme.\par
The plan of this work is as follows:\\
Dynamical properties of the model \eqref{eq:3} are studied in Section  \ref{sec2}. The dynamically consistent NSFD scheme is formulated and analyzed in Section \ref{sec3}. In Section \ref{sec4}, we report some numerical experiments to illustrate the theoretical results. Some discussions and remarks are provided in the last section.
\section{Dynamical analysis}\label{sec2}
In this section, dynamical qualitative properties of the model \eqref{eq:3} will be investigated. We first establish the positivity and boundedness of the model.
\begin{proposition}[Positivity and boundedness]\label{lemma1}
The set $\mathbb{R}_+^3$ is a positively invariant set of the model \eqref{eq:3}, i.e., $S(t), I(t), R(t) \geq 0$ whenever $S(0), I(0), R(0) \geq 0$. Furthermore, we have
\begin{equation}\label{eq:4}
\begin{split}
&\limsup_{t \to \infty}S(t) \leq \dfrac{\Lambda}{\mu_0 + \nu},\quad \dfrac{\Lambda}{p} \leq \limsup_{t \to \infty}\big(S(t) + I(t)\big) \leq \dfrac{\Lambda}{q},\quad \dfrac{r_1\Lambda}{p\mu_0} \leq \limsup_{t \to \infty}R(t) \leq \dfrac{r_2\Lambda}{q\mu_0},\\
&p := \max\{\mu_0 + \nu,\,\,\mu_0 + \mu_1 + \beta\}, \quad q := \min\{\mu_0 + \nu,\,\,\mu_0 + \mu_1 + \beta\}, \quad r_2 := \max\{\beta, \,\, \nu\}, \quad r_1 := \min\{\beta, \,\, \nu\}.
\end{split}
\end{equation}
\end{proposition}
\begin{proof}
First, the system \eqref{eq:3} implies that
\begin{equation*}
\dfrac{dS}{dt}\bigg|_{S = 0} = \Lambda > 0,\quad \dfrac{dI}{dt}\bigg|_{I = 0} = 0,\quad \dfrac{dR}{dt}\bigg|_{R = 0} = \beta I + \nu S \geq 0.
\end{equation*}
Thanks to \cite[Lemma 10]{Horvath}, the positivity of the solutions of the model \eqref{eq:3} is obtained. Hence, the set $\mathbb{R}_3^+$ is a positively invariant set of \eqref{eq:3}.\par
Next, it follows from the first equation of \eqref{eq:3} that
\begin{equation*}
\dfrac{dS}{dt} \leq \Lambda - (\mu_0 + \nu)S.
\end{equation*}
Using the basic comparison theorem for ODEs \cite{McNabb}, we obtain
\begin{equation*}
S(t) \leq \bigg(S(0) - \dfrac{\Lambda}{\mu_0 + \nu}\bigg)e^{-(\mu_0 + \nu)t} + \dfrac{\Lambda}{\mu_0 + \nu},
\end{equation*}
which implies the first estimate of \eqref{eq:4}.\par
Similarly, it follows from the first and second equation of \eqref{eq:3} that
\begin{equation*}
\Lambda - p(S + I) \leq \dfrac{d(S + I)}{dt} = \Lambda - (\mu_0 + \nu)S - (\mu_0 + \mu_1 + \beta)I \leq \Lambda - q(S + I).
\end{equation*}
Hence, the second estimate of \eqref{eq:4} is obtained.\par
Lastly, using the third equation of \eqref{eq:3} we have
\begin{equation*}
\dfrac{r_1\Lambda}{p} - \mu_0 R \leq r_1(S + I) - \mu_0 R \leq \dfrac{dR}{dt} \leq r_2(S + I) - \mu_0 R \leq \dfrac{r_2\Lambda}{q} - \mu_0 R
\end{equation*}
for $t$ large enough. Consequently, the third estimate of \eqref{eq:4} is proved. This proof is complete.
\end{proof}
\begin{proposition}[Equilibria and the basic reproduction number]\label{lemma2}
\begin{enumerate}[(i)]
\item The HBV model \eqref{eq:3} always possesses a disease-free equilibrium (DFE) point given by
\begin{equation}\label{eq:5}
E^0 = (S^0, I^0, R^0) = \bigg(\dfrac{\Lambda}{\mu_0 + \nu},\,\,0\,\,\dfrac{\nu\Lambda}{\mu_0(\mu_0 + \nu)}\bigg)
\end{equation}
for all the values of the parameters.
\item The basic reproduction number of the model \eqref{eq:3} can be computed as
\begin{equation}\label{eq:6}
\mathcal{R}_0 = \dfrac{1}{\mu_0 + \mu_1 + \beta}\dfrac{\partial f\bigg(\dfrac{\Lambda}{\mu_0 + \nu},\,\,0\bigg)}{\partial I}.
\end{equation}
\item The model \eqref{eq:3} has a unique disease-endemic equilibrium (DEE) point $E^* = (S^*, I^*, R^*)$ if and only if $\mathcal{R}_0 > 1$. Moreover, if $E^*$ exists it is given by
\begin{equation}\label{eq:7}
S^* \dfrac{\Lambda - (\mu_0 + \mu_1 + \beta)I^*}{\mu_0 + \nu}, \quad R^* = \dfrac{\beta I^* + \nu S^*}{\mu_0},
\end{equation}
where $I^*$ is a unique positive solution of the equation
\begin{equation*}
{f\bigg(\dfrac{\Lambda - (\mu_0 + \mu_1 + \beta)I}{\mu_0 + \nu},\,\, I\bigg)} - \big(\mu_0 + \mu_1 + \beta \big)I = 0. 
\end{equation*}
\end{enumerate}
\end{proposition}
\begin{proof}
\textbf{Proof of Part (i).} To determine equilibria of the model \eqref{eq:3}, we need to solve the nonlinear system of algebra equations
\begin{equation}\label{eq:8}
\Lambda - f(S, I) - \big(\mu_0 + \nu\big)S = 0,\quad f(S, I) - \big(\mu_0 + \mu_1 + \beta \big)I = 0, \quad \beta I + \nu S - \mu_0 R = 0.
\end{equation}
It follows from the first and second equations of the system \eqref{eq:8} that
\begin{equation*}
S = \dfrac{\Lambda - (\mu_0 + \mu_1 + \beta)I}{\mu_0 + \nu}.
\end{equation*}
Combining this with the second equation of \eqref{eq:8} we obtain
\begin{equation}\label{eq:9}
f\bigg(\dfrac{\Lambda - (\mu_0 + \mu_1 + \beta)I}{\mu_0 + \nu}, I\bigg) - \big(\mu_0 + \mu_1 + \beta \big)I = 0.
\end{equation}
Since $f(S, 0) = 0$ for all $S, I \geq 0$, the equation \eqref{eq:9} always has a trivial solution $I = 0$. Consequently, the system \eqref{eq:8} always possesses a solution $\big(\Lambda/(\mu_0 + \nu),\,\, 0,\,\, {\nu\Lambda}/{[\mu_0(\mu_0 + \nu)]}\big)$. This solution corresponds to a unique DFE point.\\
\textbf{Proof of Part (ii).} We apply the method proposed in \cite{vdDriessche} to compute the basic reproduction number of the model \eqref{eq:3}. Since the first two equations of \eqref{eq:3} do not depend on $R$, it is sufficient to consider the following sub-model
\begin{equation}\label{eq:10}
\begin{split}
&\dfrac{dS(t)}{dt} = \Lambda - f(S(t), I(t)) - \big(\mu_0 + \nu\big)S(t),\\
&\dfrac{dI(t)}{dt} = f(S(t), I(t)) - \big(\mu_0 + \mu_1 + \beta \big)I(t).
\end{split}
\end{equation}
We reorder the variables in \eqref{eq:10} as $(I, S)$. Then, the DFE point is transformed to $x_0 = (I_0, S_0)$. The model \eqref{eq:10} can be written in the matrix form
\begin{equation*}
\dfrac{dx}{dt} = \mathcal{F}(x) - \mathcal{V}(x),
\end{equation*}
where
\begin{equation*}
\mathcal{F}(x) = 
\begin{pmatrix}
f(S, I)\\
\Lambda
\end{pmatrix},
\quad \mathcal{V}(x) = 
\begin{pmatrix}
(\mu_0 + \mu_1 + \beta)I\\
f(S, I) + (\mu_0 + \nu)S
\end{pmatrix}
\end{equation*}
Consequently,
\begin{equation*}
D\mathcal{F}(x_0) = 
\begin{pmatrix}
\dfrac{\partial f}{\partial I}(S_0, 0)&0\\
&\\
0&0
\end{pmatrix},\quad
D\mathcal{V}(x) = 
\begin{pmatrix}
(\mu_0 + \mu_1 + \beta)&0\\
&\\
\dfrac{\partial f}{\partial I}(S_0, 0)&\dfrac{\partial f}{\partial S}(S_0, 0) + (\mu_0 + \nu).
\end{pmatrix}
\end{equation*}
Then, the basic reproduction number can be computed as
\begin{equation*}
\mathcal{R}_0 = \rho(FV^{-1}) = \dfrac{1}{\mu_0 + \mu_1 + \beta}\dfrac{\partial f\bigg(\dfrac{\Lambda}{\mu_0 + \nu},\,\,0\bigg)}{\partial I}.
\end{equation*}
\textbf{Proof of Part (iii)}. Return to the equation \eqref{eq:9}. We will prove that it has only a unique positive solution $I^* > 0$. Indeed, consider the equation
\begin{equation}\label{eq:11}
F(I) := \dfrac{f\bigg(\dfrac{\Lambda - (\mu_0 + \mu_1 + \beta)I}{\mu_0 + \nu},\,\, I\bigg)}{I} - \big(\mu_0 + \mu_1 + \beta \big) = 0.
\end{equation}
From the condition (H5), it is easy to verify that $F'(I) < 0$ for all $I > 0$. On the other hand,
\begin{equation*}
\begin{split}
&\lim_{I \to 0}F(I) = \dfrac{\partial f\bigg(\dfrac{\Lambda}{\mu_0 + \nu},\,\,0\bigg)}{\partial I} - ({\mu_0 + \mu_1 + \beta}) > 0,\\
&F\bigg(\dfrac{\Lambda}{\mu_0 + \mu_1 + \beta}\bigg) = - ({\mu_0 + \mu_1 + \beta}) < 0.
\end{split}
\end{equation*}
Hence, the equation \eqref{eq:11} has a positive solution $I^*$ and we obtain $(S^*, I^*, R^*)$ as in \eqref{eq:7}. In other words, the existence and uniqueness of the DEE point are proved.
\end{proof}
\begin{lemma}[Local asymptotic stability]\label{lemma3}
\begin{enumerate}[(i)]
\item The DFE point $E^0$ is locally asymptotically stable if  $\mathcal{R}_0 < 1$ and unstable if $\mathcal{R}_0 > 1$.
\item The DEE point $E^*$ is locally asymptotically stable if  $\mathcal{R}_0 > 1$ and unstable if $\mathcal{R}_0 < 1$.
\end{enumerate}
\end{lemma}
\begin{proof}
\textbf{Proof of Part (i).} The Jacobian matrix of the system \eqref{eq:3} at $E^0$ is given by
\begin{equation*}
J(E^0) =
\begin{pmatrix}
-(\mu_0 + \nu)&-\dfrac{\partial f(E^0)}{\partial I}&0\\
&&\\
0&\dfrac{\partial f(E^0)}{\partial I} - (\mu_0 + \mu_1 + \beta)&0\\
&&\\
\nu&\beta&-\mu_0
\end{pmatrix}.
\end{equation*}
Hence, three eigenvalues of $J(E^0)$ are
\begin{equation*}
\lambda_1 = -(\mu_0 + \nu), \quad \lambda_2 = \dfrac{\partial f(E^0)}{\partial I} - (\mu_0 + \mu_1 + \beta), \quad \lambda_3 = -\mu_0.
\end{equation*}
It is clear that $\lambda_1 < 0$ and $\lambda_3 < 0$. Moreover, if $\mathcal{R}_0 < 1$ then $\lambda_3 < 0$. By the Lyapunov direct method \cite{Allen, Martcheva}, the DFE point $E^0$ is locally asymptotically stable. Conversely, if $\mathcal{R}_0 > 1$ then $\lambda_3 > 0$ and hence, $E^0$ is unstable. The proof is complete.\\
\textbf{Proof of Part (ii).}  Note that $E^*$ exists if and only if $\mathcal{R}_0 > 1$. The Jacobian matrix of the system \eqref{eq:3} at $E^*$ is given by
\begin{equation*}
J(E^*) =
\begin{pmatrix}
-\dfrac{\partial f(E^*)}{\partial S} - (\mu_0 + \nu)&-\dfrac{\partial f(E^*)}{\partial I}&0\\
&&\\
\dfrac{\partial f(E^*)}{\partial S}&\dfrac{\partial f(E^*)}{\partial I} - (\mu_0 + \mu_1 + \beta)&0\\
&&\\
\nu&\beta&-\mu_0
\end{pmatrix}.
\end{equation*}
Hence, one of three eigenvalues of $J(E^*)$ is $\lambda_1 = -\mu_0 < 0$ and two remaining eigenvalues are the ones of the sub-matrix
\begin{equation*}
\widehat{J}(E^*) =
\begin{pmatrix}
-\dfrac{\partial f(E^*)}{\partial S} - (\mu_0 + \nu)&-\dfrac{\partial f(E^*)}{\partial I}\\
&&\\
\dfrac{\partial f(E^*)}{\partial S}&\dfrac{\partial f(E^*)}{\partial I} - (\mu_0 + \mu_1 + \beta)
\end{pmatrix}.
\end{equation*}
Because $E^*$ is the positive equilibrium point we have
\begin{equation*}
\dfrac{f(S^*, I^*)}{I^*} = (\mu_0 + \mu_1 + \beta).
\end{equation*}
Hence, it follows from the hypothesis (H5) that
\begin{equation}\label{eq:12new}
\dfrac{\partial f(S^*, I^*)}{\partial I} < \dfrac{f(S^*, I^*)}{I^*} = (\mu_0 + \mu_1 + \beta).
\end{equation}
From this estimate, we obtain
\begin{equation*}\label{eq:12}
Trace(\widehat{J}) < 0, \quad \det(\widehat{J}) > 0.
\end{equation*}
Thanks to the Lyapunov stability theorem \cite{Allen, Martcheva}, the local asymptotic stability of $E^*$ is proved. This proof is completed.
\end{proof}
To end this section, we investigate the GAS of the model \eqref{eq:3}. As an consequence of Proposition \ref{lemma1}, it is sufficient to consider the GAS of the model \eqref{eq:3} on its feasible set defined by
\begin{equation*}\label{eq:13}
\Omega = \bigg\{(S, I, R)\Big|S, I, R \geq 0,\,\, 0 \leq S \leq \dfrac{\Lambda}{\mu_0 + \nu},\,\, \dfrac{\Lambda}{p} \leq S  + I \leq \dfrac{\Lambda}{q},\,\, \dfrac{r_1\Lambda}{p\mu_0} \leq R \leq \dfrac{r_2\Lambda}{q\mu_0}\bigg\},
\end{equation*}
where $p, q, r_1$ and $r_2$ are given by \eqref{eq:4}. Moreover, since the first two equations of the model \eqref{eq:3} do not depend on $R$, we only need to consider the GAS of the sub-model
\begin{equation}\label{eq:14}
\begin{split}
&\dfrac{dS}{dt} = \Lambda - f(S, I) - \big(\mu_0 + \nu\big)S,\\
&\dfrac{dI}{dt} = f(S, I) - \big(\mu_0 + \mu_1 + \beta \big)I
\end{split}
\end{equation}
on its feasible set
\begin{equation*}\label{eq:15}
\Omega^* = \bigg\{(S, I)\big|S, I \geq 0,\,\, 0 \leq S \leq \dfrac{\Lambda}{\mu_0 + \nu},\,\, \dfrac{\Lambda}{p} \leq S  + I \leq \dfrac{\Lambda}{q}\bigg\}.
\end{equation*}
On the set $\Omega^*$, the model \eqref{eq:14} always has a DFE point $\widehat{E}^0 = (S^0, 0)$ for all the values of the parameters and a DEE point $\widehat{E}^* = (S^*, I^*)$ exists if and only if $\mathcal{R}_0 > 1$. 
\begin{lemma}\label{lemma4}
Let $f(S, I)$ be a function satisfying the conditions (H1)-(H5). We define
\begin{equation}\label{eq:16}
V(S) := S - S_0 - \int_{S_0}^S \dfrac{f(S_0, I)}{f(t, I)}dt, \quad S \geq 0
\end{equation}
and
\begin{equation}\label{eq:17}
g(I) = \dfrac{f(S_0, I)}{I} - (\mu_0 + \mu_1 + \nu), \quad I > 0.
\end{equation}
Then,
\begin{enumerate}[(i)]
\item $V(S) \geq 0$ for all $S \geq 0$ and $V(S) = 0$ if and only if $S = S_0$.
\item If $\mathcal{R}_0 \leq 1$, then $g(I) < 0$ for all $I > 0$.
\end{enumerate}
\end{lemma}
\begin{proof}
\textbf{Proof of Part (i).} First, it is clear that $V(S_0) = 0$. The derivative of the function $V$ is given by
\begin{equation*}
V'(S) = 1 - \dfrac{f(S_0, I)}{f(S, I)}.
\end{equation*}
From the condition (H3) we deduce that
\begin{equation*}
\begin{split}
&V'(S) > 0\quad \mbox{for all} \quad S > S_0,\\
&V'(S) < 0\quad \mbox{for all} \quad S < S_0,\\
&V'(S) = 0\quad \mbox{if and only if}\quad S = S_0. 
\end{split}
\end{equation*}
Hence, $V(S) > V(S_0) = 0$ for all $S \ne S_0$. The proof of this part is complete.\\
\textbf{Proof of Part (ii).} It follows from the condition (H5) that
\begin{equation*}
g'(I) = \dfrac{\partial f(S_0, I)}{\partial I} < 0.
\end{equation*}
On the other hand,
\begin{equation*}
\lim_{I \to 0}g(I) = \dfrac{\partial f\bigg(\dfrac{\Lambda}{\mu_0 + \nu},\,\,0\bigg)}{\partial I} - (\mu_0 + \mu_1 + \nu) = (\mu_0 + \mu_1 + \nu)(\mathcal{R}_0 - 1) \leq 0.
\end{equation*}
Therefore, we conclude that $g(I) < (\mu_0 + \mu_1 + \nu)(\mathcal{R}_0 - 1) \leq 0$. This completes the proof.
\end{proof}
\begin{theorem}\label{theorem1}
The DFE point $\widehat{E^0}$ of the model \eqref{eq:14} is globally asymptotically stable if $\mathcal{R}_0 \leq 1$.
\end{theorem}
\begin{proof}
Consider a Lyapunov function given by
\begin{equation}\label{eq:18}
L(S, I) = V(S) + I,
\end{equation}
where $V(S)$ is defined by \eqref{eq:16}. The derivative of the function $L$ along solutions of \eqref{eq:14} is given by
\begin{equation*}
\begin{split}
\dfrac{dL}{dt} = \dfrac{dL}{dS}\dfrac{dS}{dt} + \dfrac{dL}{dI}\dfrac{dI}{dt}& = \bigg(1 - \dfrac{f(S_0, I)}{f(S, I)}\bigg)\Big[\Lambda - f(S, I) - \big(\mu_0 + \nu\big)S\Big] + \Big[f(S, I) - \big(\mu_0 + \mu_1 + \beta \big)I\Big]\\
&=  \bigg(1 - \dfrac{f(S_0, I)}{f(S, I)}\bigg)\Big[\Lambda - \big(\mu_0 + \nu\big)S\Big] + \Big[f(S_0, I)- \big(\mu_0 + \mu_1 + \beta \big)I\Big]\\
&= \bigg(1 - \dfrac{f(S_0, I)}{f(S, I)}\bigg)\big(\mu_0 + \nu\big)\big(S_0 - S\big) + Ig(I),
\end{split}
\end{equation*}
where the function $g(I)$ is defined by \eqref{eq:17}.\par
Combining the properties of the function $f(S, I)$ and Lemma \ref{lemma2}, we deduce that the function $L$ given by \eqref{eq:18} satisfies the Lyapunov direct method \cite{LaSalle, Lyapunov}. So, the GAS of $\widehat{E^0}$ is obtained. The proof is complete.
\end{proof}
\begin{theorem}\label{theorem2}
Suppose that $\mathcal{R}_0 > 1$. Then, the DEE point $\widehat{E^*}$ of the model \eqref{eq:14} is globally asymptotically stable whenever $I(0) > 0$.
\end{theorem}
\begin{proof}
Note that if $I(0) = 0$ then $\big(S(t), I(t)\big)$ will converge to the DFE point $\widehat{E^0}$ as $t \to \infty$. We first use the Dulac-Bendixson criterion \cite[Theorem 3.6]{Martcheva} to show that the system \eqref{eq:14} has no periodic orbits or graphics in $\mathbb{R}_2^+$. Indeed, let us  denote by $F(S, I)$ and $G(S, I)$ be the right-hand side functions of \eqref{eq:14}. Consider a function 
\begin{equation*}
D(S, I) = \dfrac{1}{I}.
\end{equation*}
We have
\begin{equation*}
\begin{split}
\dfrac{\partial(DF)}{\partial S} + \dfrac{\partial(DG)}{\partial I} &= \dfrac{\partial}{\partial S}\bigg[\dfrac{\Lambda}{I} - \dfrac{f(S, I)}{I} - \dfrac{(\mu_0 + \nu)S}{I}\bigg] + \dfrac{\partial}{\partial I}\bigg[\dfrac{f(S, I)}{I} - (\mu_0 + \mu_1 + \beta)\bigg]\\
&= -\dfrac{1}{I}\dfrac{\partial f(S, I)}{S} - \dfrac{\mu_0 + \nu}{I} + \dfrac{\partial}{\partial I}\bigg(\dfrac{f(S, I)}{I}\bigg).
\end{split}
\end{equation*}
It follows from the conditions (H3) and (H5) that 
\begin{equation*}
\dfrac{\partial(DF)}{\partial S} + \dfrac{\partial(DG)}{\partial I} < 0.
\end{equation*}
Hence, the system \eqref{eq:14} has no periodic orbits or graphics in the open first quadrant.\par
Since all solutions of the model \eqref{eq:14} are bounded and the DFE point $\widehat{E^0}$ is unstable saddle if $\mathcal{R}_0 > 1$, using Poincar\'e-Bendixson theorem \cite[Theorem 3.5]{Martcheva} we obtain the GAS of the DEE point $\widehat{E^*}$. The proof is complete.
\end{proof}
From Theorems \ref{theorem1} and \ref{theorem2}, we obtain the GAS of the model \eqref{eq:3} as follows.
\begin{theorem}\label{theorem3}
\begin{enumerate}[(i)]
\item The DFE point ${E^0}$ of the model \eqref{eq:3} is not only locally asymptotically stable but also globally asymptotically stable if $\mathcal{R}_0 \leq 1$.
\item The DEE point ${E^*}$ of the model \eqref{eq:3} is not only locally asymptotically but also globally asymptotically stable provided that $\mathcal{R}_0 > 1$ and $I(0) > 0$.
\end{enumerate}
\end{theorem}
\section{Construction of dynamically consistent NSFD scheme}\label{sec3}
In this section, we construct and analyze a dynamically consistent NSFD scheme for the model \eqref{eq:3}. For this purpose, consider the model \eqref{eq:3} on a finite time interval $[0, T]$ and discretize this interval by
\begin{equation*}
0 = t_0 < t_1 < \ldots < t_{N - 1} < t_N = T,
\end{equation*}
where $t_{n + 1} - t_n = \Delta t$ for $0 \leq n \leq N - 1$. Let us denote by $(S_n, I_n, R_n)$ the intended approximation for $\big(S(t_n), I(t_n), R(t_n)\big)$ for $n = 1, 2, \ldots, N $. Applying the Mickens' methodology \cite{Mickens1, Mickens2, Mickens3, Mickens4, Mickens5} we approximate the first derivatives in the model \eqref{eq:3} by
\begin{equation}\label{eq:20}
\dfrac{dS}{dt}\bigg|_{t = t_n} \approx \dfrac{S_{n + 1} - S_n}{\varphi(\Delta t)},\qquad \dfrac{dI}{dt}\bigg|_{t = t_n} \approx \dfrac{I_{n + 1} - I_n}{\varphi(\Delta t)},\qquad \dfrac{dR}{dt}\bigg|_{t = t_n} \approx \dfrac{R_{n + 1} - R_n}{\varphi(\Delta t)},
\end{equation}
where $\varphi$, called the denominator function, is a positive function satisfying $\varphi(\Delta t) = \Delta t + \mathcal{O}(\Delta t^2)$ as $\Delta t \to 0$. Next, we discretize the right-hand side function of the model \eqref{eq:3} as follows. Let us denote
\begin{equation}\label{eq:21}
f^*(S, I) = 
\begin{cases}
&\dfrac{f(S, I)}{S},\quad S > 0,\\
&0,\quad S = 0.
\end{cases}
\end{equation}
Then, $f(S, I) = Sf^*(S, I)$. Moreover, it follows from the conditions (H1) and (H5) that $f^*(S, I)$ is differentiable for $S, I \geq 0$. In particular,
\begin{equation}\label{eq:21a}
\begin{split}
&f^*(S^0, I^0) = f^*(S^*, I^*) = 0,\\
&\dfrac{\partial f^*}{\partial S}(S_0, 0) = 0,\quad \dfrac{\partial f^*}{\partial I}(S_0, 0) = \dfrac{1}{S_0}\dfrac{\partial f}{\partial I}(S_0, 0),\\
&\dfrac{\partial f^*}{\partial S}(S^*, I^*) = \dfrac{1}{S^*}\dfrac{\partial f}{\partial S}(S^*, I^*), \quad \dfrac{\partial f^*}{\partial I}(S^*, I^*) = \dfrac{1}{S^*}\dfrac{\partial f}{\partial I}(S^*, I^*).
\end{split}
\end{equation}
We now discretize the right-hand side function of \eqref{eq:3} by
\begin{equation}\label{eq:22}
\begin{split}
&\Lambda - f(S(t_n), I(t_n)) - \big(\mu_0 + \nu\big)S(t_n) \approx \Lambda - S_{n+1}f^*(S_n, I_n) -\big(\mu_0 + \nu\big)S_{n+1} \\
&f(S(t_n), I(t_n)) - \big(\mu_0 + \mu_1 + \beta \big)I(t_n) \approx S_{n+1}f^*(S_n, I_n) - \big(\mu_0 + \mu_1 + \beta \big)I_{n + 1},\\
&\beta I(t_n) + \nu S(t_n) - \mu_0 R(t_n) \approx \beta I_{n + 1} + \nu S_{n + 1} - \mu_0 R_{n + 1}.
\end{split}
\end{equation}
The systems \eqref{eq:20} and \eqref{eq:22} lead to the the following NSFD model for the model \eqref{eq:3}
\begin{equation}\label{eq:23}
\begin{split}
&\dfrac{S_{n + 1} - S_n}{\varphi(\Delta t)} = \Lambda - S_{n+1}f^*(S_n, I_n) -\big(\mu_0 + \nu\big)S_{n+1},\\
&\dfrac{I_{n + 1} - I_n}{\varphi(\Delta t)} = S_{n+1}f^*(S_n, I_n) - \big(\mu_0 + \mu_1 + \beta \big)I_{n + 1},\\
&\dfrac{R_{n + 1} - R_n}{\varphi(\Delta t)} = \beta I_{n + 1} + \nu S_{n + 1} - \mu_0 R_{n + 1}.
\end{split}
\end{equation}
\begin{theorem}\label{theorem4}
The set $\mathbb{R}^3_+$ is a positively invariant set of the NSFD model \eqref{eq:3}, that is, $S_n, I_n, R_n \geq 0$ whenever $S_0, I_0, R_0 \geq 0$. Moreover, we have
\begin{equation}\label{eq:24}
\begin{split}
&\limsup_{n \to \infty}S_n \leq \dfrac{\Lambda}{\mu_0 + \nu},\quad \quad \dfrac{\Lambda}{p} \leq \limsup_{n \to \infty}\big(S_n + I_n\big) \leq \dfrac{\Lambda}{q},\quad \dfrac{r_1\Lambda}{p\mu_0} \leq \limsup_{n \to \infty}R_n \leq \dfrac{r_2\Lambda}{q\mu_0},
\end{split}
\end{equation}
where $p, q, r_1$ and $r_2$ are given by \eqref{eq:4}.
\end{theorem}
\begin{proof}
First, we rewrite the system \eqref{eq:23} in the explicit form
\begin{equation}\label{eq:25}
\begin{split}
&S_{n + 1} = \dfrac{S_n + \varphi \Lambda}{1 + \varphi f^*(S_n, I_n) + \varphi(\mu_0 + \nu)},\\
&I_{n + 1} = \dfrac{I_n + \varphi S_{n + 1}f^*(S_n, I_n)}{1 + \varphi(\mu_0 + \mu_1 + \beta)},\\
&R_{n + 1} = \dfrac{R_n + \varphi \beta I_{n + 1} + \varphi\nu S_{n + 1}}{1 + \varphi \mu_0}.
\end{split}
\end{equation}
Hence, $S_{n + 1}, I_{n + 1}, R_{n + 1} \geq 0$ if $S_n, I_n, R_n \geq 0$.\par
It follows from the first equation of \eqref{eq:23} that
\begin{equation*}
\dfrac{S_{n + 1} - S_n}{\varphi} \leq \Lambda - (\mu_0 + \nu)S_{n + 1},
\end{equation*}
which implies that
\begin{equation*}
\begin{split}
S_{n + 1} &\leq \dfrac{S_n}{1 + \varphi(\mu_0 + \nu)} + \dfrac{\varphi\Lambda}{1 + \varphi(\mu_0 + \nu)}\\
&\leq \dfrac{1}{1 + \varphi(\mu_0 + \nu)}\bigg[ \dfrac{S_{n - 1}}{1 + \varphi(\mu_0 + \nu)} + \dfrac{\varphi\Lambda}{1 + \varphi(\mu_0 + \nu)}\bigg] + \dfrac{\varphi\Lambda}{1 + \varphi(\mu_0 + \nu)}\\
&= \bigg[\dfrac{1}{1 + \varphi(\mu_0 + \nu)}\bigg]^2S_{n - 1} + \dfrac{\varphi\Lambda}{1 + \varphi(\mu_0 + \nu)}\bigg[1 + \dfrac{1}{1 + \varphi(\mu_0 + \nu)}\bigg]\\
&\leq \ldots \leq \bigg[\dfrac{1}{1 + \varphi(\mu_0 + \nu)}\bigg]^{n + 1}S_{0} + \dfrac{\varphi\Lambda}{1 + \varphi(\mu_0 + \nu)}\sum_{j = 0}^n\bigg(\dfrac{1}{1 + \varphi(\mu_0 + \nu)}\bigg)^j\\
&=  \bigg[\dfrac{1}{1 + \varphi(\mu_0 + \nu)}\bigg]^{n + 1}S_{0} +\dfrac{\varphi\Lambda}{1 + \varphi(\mu_0 + \nu)}\dfrac{1 - \bigg(\dfrac{1}{1 + \varphi(\mu_0 + \nu)}\bigg)^{n + 1}}{1 - \dfrac{1}{1 + \varphi(\mu_0 + \nu)}}.
\end{split}
\end{equation*}
Letting $n \to \infty$, the first estimate of  \eqref{eq:24} is obtained.\par
Similarly, from the first and second equations of \eqref{eq:23} we have
\begin{equation*}
\Lambda - p(S_{n + 1} + I_{n + 1}) \leq \dfrac{(S_{n + 1} + I_{n+1}) - (S_n + I_n)}{\varphi} = \Lambda - (\mu_0 + \nu)S_{n + 1} - (\mu_0 + \mu_1 + \beta)I_{n + 1} \leq \Lambda - q(S_{n + 1} + I_{n + 1}),
\end{equation*}
which implies that
\begin{equation*}
\dfrac{\varphi\Lambda + (S_n + I_n)}{1 + \varphi p} \leq (S_{n + 1} + I_{n + 1}) \leq \dfrac{\varphi\Lambda + (S_n + I_n)}{1 + \varphi q}.
\end{equation*}
From this estimate, we obtain the second estimate of \eqref{eq:24}.\par
Lastly, it follows from the third equation of \eqref{eq:23} that
\begin{equation*}
r_1\dfrac{\Lambda}{p} - \mu_0 R_{n+1} \leq r_1(S_{n + 1} + I_{n + 1}) - \mu_0 R_{n+1} \leq \dfrac{R_{n + 1} - R_n}{\varphi} \leq r_2(S_{n + 1} + I_{n + 1}) - \mu_0 R_{n + 1} \leq r_2\dfrac{\Lambda}{q} - \mu_0 R_{n+1}
\end{equation*}
for $n$ large enough. Hence,
\begin{equation*}
\dfrac{R_n + \varphi r_1{\Lambda}/{p}}{1 + \varphi\mu_0} \leq R_{n + 1} \leq \dfrac{R_n + \varphi r_2{\Lambda}/{q}}{1 + \varphi\mu_0},
\end{equation*}
which implies the last estimate of \eqref{eq:24}. The proof is completed.
\end{proof}
We now compute the basic reproduction number for the discrete model \eqref{eq:23} by the next generation matrix approach \cite{Allen1}.\par
It is easy to verify that the model \eqref{eq:23} always has a unique DFE point $E^0 = (S_0, 0, R_0)$ for all the values of the parameters. Since the first two equations of \eqref{eq:23} do not include $R_n$, it is sufficient to consider the following sub-model
\begin{equation}\label{eq:26}
\begin{split}
&\dfrac{S_{n + 1} - S_n}{\varphi(\Delta t)} = \Lambda - S_{n+1}f^*(S_n, I_n) -\big(\mu_0 + \nu\big)S_{n+1},\\
&\dfrac{I_{n + 1} - I_n}{\varphi(\Delta t)} = S_{n+1}f^*(S_n, I_n) - \big(\mu_0 + \mu_1 + \beta \big)I_{n + 1}.
\end{split}
\end{equation}
We reorder the variables in \eqref{eq:26} as $(I_n, S_n)$. Then, the DFE point is transformed to $(0, S_0)$. By using \eqref{eq:21a}, the corresponding Jacobian matrix of \eqref{eq:26} at $(0, S_0)$ is
\begin{equation*}
J_0 = 
\begin{pmatrix}
\dfrac{1 + \varphi \dfrac{\partial f}{\partial I}(S_0, 0)}{1 + \varphi(\mu_0 + \mu_1 + \beta)}&0\\
&\\
-\dfrac{\varphi\dfrac{\partial f}{\partial I}(S_0, 0)}{1 + \varphi(\mu_0 + \nu)}&\dfrac{1}{1 + \varphi(\mu_0 + \nu)}
\end{pmatrix}
\end{equation*}
Following the method in \cite{Allen1}, we represent the matrix $J_0$ in the form
\begin{equation*}
J_0 = 
\begin{pmatrix}
F + T&0\\
A&C
\end{pmatrix},
\end{equation*}
where
\begin{equation*}
F = \dfrac{\varphi \dfrac{\partial f}{\partial I}(S_0, 0)}{1 + \varphi(\mu_0 + \mu_1 + \beta)}, \quad T = \dfrac{1}{1 + \varphi(\mu_0 + \mu_1 + \beta)}, \quad C = \dfrac{1}{1 + \varphi(\mu_0 + \nu}, \quad A = -\dfrac{\varphi\dfrac{\partial f}{\partial I}(S_0, 0)}{1 + \varphi(\mu_0 + \nu)}.
\end{equation*}
It is clear that  $F$ and $T$ are non-negative, $F + T$ is irreducible, and matrices $C$ and $T$ satisfy
\begin{equation*}
\rho(T) < 1, \quad \rho(C) < 1.
\end{equation*}
Hence, the basic reproduction number of the discrete model \eqref{eq:23} can be computed as
\begin{equation*}
\mathcal{R}_0 = \rho\big(F(1 - T)^{-1}\big) = \dfrac{1}{\mu_0 + \mu_1 + \beta}\dfrac{\partial f\bigg(\dfrac{\Lambda}{\mu_0 + \nu},\,\,0\bigg)}{\partial I},
\end{equation*}
which means that the basic reproduction numbers of the models \eqref{eq:23} and \eqref{eq:3} are equal.\par
Similarly to Proposition \ref{lemma2}, it is easy to verify that the NSFD model \eqref{eq:23} has a unique DEE point $E^* = (S^*, I^*, R^*)$ if and only if $\mathcal{R}_0 > 1$.\par
The following result on the local asymptotic stability of the DFE point of the model \eqref{eq:23} is a direct consequence of Theorem 2.1 in \cite{Allen1}.
\begin{corollary}[Local asymptotic stability of the DFE point]\label{corollary1}
The DFE point $E^0$ of the NSFD model \eqref{eq:23} is locally asymptotically stable if $\mathcal{R}_0 \leq 1$ and unstable if $\mathcal{R}_0 > 1$.
\end{corollary}
\begin{proof}
From Theorem 2.1 in \cite{Allen1}, we conclude that the DFE point $(S_0, I_0)$  of \eqref{eq:26} is locally asymptotically stable if $\mathcal{R}_0 < 1$ and unstable if $\mathcal{R}_0 > 1$. Consequently, the local asymptotic stability of $E^0$ of \eqref{eq:23} is proved.
\end{proof}
\begin{theorem}\label{theorem5}
The DEE point $E^*$ of the NSFD model \eqref{eq:23} is locally asymptotically stable if $\mathcal{R}_0 > 1$.
\end{theorem}
\begin{proof}
The Jacobian matrix of the system \eqref{eq:23} at $E^*$ is given by
\begin{equation*}\label{eq:29}
J(E^*) = 
\begin{pmatrix}
\dfrac{1 - \varphi\dfrac{\partial f}{\partial S}(S^*, I^*)}{1 + \varphi(\mu_0 + \nu)}&-\dfrac{\varphi\dfrac{\partial f}{\partial I}(S^*, I^*)}{1 + \varphi(\mu_0 + \nu)}&0\\
&\\
\dfrac{\varphi\dfrac{\partial f}{\partial S}(S^*, I^*)}{1 + \varphi(\mu_0 + \mu_1 + \beta)}&\dfrac{1 + \varphi\dfrac{\partial f}{\partial I}(S^*, I^*)}{1 + \varphi(\mu_0 + \mu_1 + \beta)}&0\\
&\\
\dfrac{\partial R_{n+1}}{\partial S_n}(E^*)&\dfrac{\partial R_{n+1}}{\partial I_n}(E^*)&\dfrac{1}{1 + \varphi\mu_0}
\end{pmatrix}.
\end{equation*}
Hence, one of three eigenvalue of $J(E^*)$ is $\lambda_3 = \dfrac{1}{1 + \varphi\mu_0}$ and two remaining eigenvalues are the ones of the sub-matrix
\begin{equation*}\label{eq:30}
J^* = 
\begin{pmatrix}
\dfrac{1 - \varphi\dfrac{\partial f}{\partial S}(S^*, I^*)}{1 + \varphi(\mu_0 + \nu)}&-\dfrac{\varphi\dfrac{\partial f}{\partial I}(S^*, I^*)}{1 + \varphi(\mu_0 + \nu)}\\
&\\
\dfrac{\varphi\dfrac{\partial f}{\partial S}(S^*, I^*)}{1 + \varphi(\mu_0 + \mu_1 + \beta)}&\dfrac{1 + \varphi\dfrac{\partial f}{\partial I}(S^*, I^*)}{1 + \varphi(\mu_0 + \mu_1 + \beta)}
\end{pmatrix}.
\end{equation*}
It is clear that $|\lambda_3| < 1$. Denote by $\lambda_1$ and $\lambda_2$ the eigenvalues of $J^*$. We will show that $|\lambda_1| < 1$ and $|\lambda_2| < 1$. Thanks to Theorem 2.10 in \cite{Allen}, it is sufficient to verify that
\begin{equation}\label{eq:31}
|Tr(J^*)| < 1 + \det(J^*) < 2.
\end{equation}
Indeed, we first have
\begin{equation*}
\det(J^*) = \dfrac{1 + \varphi\dfrac{\partial f}{\partial I}(S^*, I^*) - \varphi\dfrac{\partial f}{\partial S}(S^*, I^*)}{\big[1 + \varphi(\mu_0 + \mu_1 + \beta)\big]\big[1 + \varphi(\mu_0 + \nu)\big]}.
\end{equation*}
So, it follows from \eqref{eq:12new} that $\det(J^*) < 1$. On the other hand,
\begin{equation*}
Tr(J^*) = \dfrac{\bigg(1 - \varphi\dfrac{\partial f}{\partial S}(S^*, I^*)\bigg)\big[1 + \varphi(\mu_0 + \mu_1 + \beta)\big] + \bigg(1 + \varphi\dfrac{\partial f}{\partial I}(S^*, I^*)\bigg)\big[1 + \varphi(\mu_0 + \nu)\big]}{\big[1 + \varphi(\mu_0 + \mu_1 + \beta)\big]\big[1 + \varphi(\mu_0 + \nu)\big]}.
\end{equation*}
Hence, it is easy to verify that
\begin{equation*}
1 + \det(J^*) - Tr(J^*) > 0, \quad 1 + \det(J^*) + Tr(J^*) > 0,
\end{equation*}
or equivalently,
\begin{equation*}
|Tr(J^*)| < 1 + \det(J^*).
\end{equation*}
Therefore, \eqref{eq:31} is verified. This completes the proof.
\end{proof}
\begin{remark}\label{remark1}
Summing up the results constructed in this section, we conclude that the NSFD model \eqref{eq:23} is dynamically consistent to the continuous model \eqref{eq:3}, with respect to the positivity, boundedness, basic reproduction number and local asymptotic stability. Also, as will be suggested by numerical examples in the next section, the GAS of the model \eqref{eq:3} is also preserved by the NSFD model \eqref{eq:23}.
\end{remark}
To end this section, we investigate the convergence and error bounds for the NSFD scheme \eqref{eq:23}. We recall that a difference scheme is said to be \textit{convergent of order p} if the global error $e_n$ with $e_0 = 0$ and $e_n := y_n - y(t_n)$, satisfies (see \cite[Section 3.2]{Ascher})
\begin{equation*}
e_n = \mathcal{O}(\Delta t^p), \quad n = 1, 2, 3,\ldots.
\end{equation*}
\begin{theorem}\label{theorem6}
The NSFD scheme \eqref{eq:23} is convergent of order $1$.
\end{theorem}
\begin{proof}
Let $\big(S(0), I(0), R(0)\big) \in \mathbb{R}_3^+$ be any  initial data for initial value problem \eqref{eq:3}. From the boundedness of \eqref{eq:3} and \eqref{eq:23} we set
\begin{equation*}
\begin{split}
&C_S := \sup_{t \geq 0}S(t), \quad C_I := \sup_{t \geq 0}I(t), \quad C_R := \sup_{t \geq 0}R(t), \quad D_S := \sup_{n \geq 0}\{S_n\}, \quad D_I := \sup_{n \geq 0}\{I_n\},\quad D_R := \sup_{n \geq 0}\{R_n\},\\
&\Omega^{CD} := \Big\{(S, I, R)\big|0 \leq S \leq \max\{C_S, D_S\},\,\, 0 \leq I \leq \max\{C_I, D_I\},\,\, 0 \leq R \leq \max\{C_R, D_R\}\Big\}.
\end{split}
\end{equation*}
For each arbitrary but fixed step size $\Delta t^* > 0$, we denote
\begin{equation*}
\Omega^D := \Big\{(S, I, R)\big| 0 \leq S \leq D_S,\,\, 0 \leq I \leq D_I, \,\, 0 \leq R \leq D_R\Big\}.
\end{equation*}
Next, we denote by $F_1(S, I, R), F_2(S, I, R)$ and $F_3(S, I, R)$ the 1st, 2nd and 3rd equations of the right-side functions of the model \eqref{eq:3}. By the boundedness of $S(t), I(t), R(t)$ and the continuity of $g_1, g_2, g_3$, it is valid to set
\begin{equation}\label{eq:32}
L_S := \sup_{t \geq 0}|S''(t)|, \quad L_I := \sup_{t \geq 0}|I''(t)|, \quad L_R := \sup_{t \geq 0}|R(t)|, \quad L := M_S + M_I + M_R
\end{equation}
and
\begin{equation}\label{eq:33}
\begin{split}
&m_j^S := \max_{(S, I, R) \in \Omega^{CS}}\bigg|\dfrac{\partial F_j}{\partial S}(S, I, R)\bigg|, \quad j = 1, 2, 3,\\
&m_j^I := \max_{(S, I, R) \in \Omega^{CS}}\bigg|\dfrac{\partial F_j}{\partial I}(S, I, R)\bigg|, \quad j = 1, 2, 3,\\
&m_j^R := \max_{(S, I, R) \in \Omega^{CS}}\bigg|\dfrac{\partial F_j}{\partial R}(S, I, R)\bigg|, \quad j = 1, 2, 3,\\
&m := \max\Big\{\sum_{j = 1}^3 m_j^S,,\,\sum_{j = 1}^3 m_j^I,\,\, \sum_{j = 1}^3 m_j^R\Big\}.
\end{split}
\end{equation}
Let us denote by $G_1(S_n, I_n, R_n, \Delta t), G_2(S_n, I_n, R_n, \Delta t)$ and $G_3(S_n, I_n, R_n, \Delta t)$ the 1st, 2nd and 3rd equations of the system \eqref{eq:25}, respectively. It is easy to verify that
\begin{equation}\label{eq:34}
\begin{split}
&G_1(S_n, I_n, R_n, 0) \equiv S_n,\\
&G_2(S_n, I_n, R_n, 0) \equiv I_n,\\
&G_3(S_n, I_n, R_n, 0) \equiv R_n,\\
&\dfrac{\partial G_j}{\partial \Delta t}(S_n, I_n, R_n) \equiv F_j(S_n, I_n, R_n), \quad j = 1, 2, 3.
\end{split}
\end{equation}
Assume that $\varphi''(\Delta t)$ exists for all $\Delta t \geq 0$. We set
\begin{equation}\label{eq:35}
\begin{split}
&M_1 := \max_{(S_n, I_n, R_n) \in \Omega^D}\bigg|\dfrac{\partial^2 G_1}{\partial \Delta t^2}(S_n, I_n, R_n, \Delta t)\bigg|,\\
&M_2 := \max_{(S_n, I_n, R_n) \in \Omega^D}\bigg|\dfrac{\partial^2 G_2}{\partial \Delta t^2}(S_n, I_n, R_n, \Delta t)\bigg|,\\
&M_3 := \max_{(S_n, I_n, R_n) \in \Omega^D}\bigg|\dfrac{\partial^2 G_3}{\partial \Delta t^2}(S_n, I_n, R_n, \Delta t)\bigg|,\\
&M := M_1 + M_2 + M_3.
\end{split}
\end{equation}
Now, using the Taylor's theorem we obtain
\begin{equation}\label{eq:36}
\begin{split}
&S(t_{n+1}) = S(t_n) + \Delta t F_1(S(t_n), I(t_n), R(t_n)) + \dfrac{\Delta t^2}{2}S''(\xi_S), \quad t_n < \xi_S < t_{n + 1},\\
&I(t_{n+1}) = I(t_n) + \Delta t F_2(S(t_n), I(t_n), R(t_n)) + \dfrac{\Delta t^2}{2}I''(\xi_I), \quad t_n < \xi_I < t_{n + 1},\\
&R(t_{n+1}) = R(t_n) + \Delta t F_3(S(t_n), I(t_n), R(t_n)) + \dfrac{\Delta t^2}{2}S''(\xi_R), \quad t_n < \xi_R < t_{n + 1},\\
\end{split}
\end{equation}
Next, it follows from the Taylor's theorem and \eqref{eq:34} that
\begin{equation}\label{eq:37}
\begin{split}
&S_{n + 1} = S_n + \Delta t F_1(S_n, I_n, R_n) + \dfrac{\Delta t^2}{2}\dfrac{\partial^2 G_1}{\partial \Delta t^2}(S_n, I_n, R_n, \Delta t_S), \quad 0 < \Delta t_S < \Delta t,\\
&I_{n + 1} = I_n + \Delta t F_2(S_n, I_n, R_n) + \dfrac{\Delta t^2}{2}\dfrac{\partial^2 G_2}{\partial \Delta t^2}(S_n, I_n, R_n, \Delta t_I), \quad 0 < \Delta t_I < \Delta t,\\
&R_{n + 1} = R_n + \Delta t F_3(S_n, I_n, R_n) + \dfrac{\Delta t^2}{2}\dfrac{\partial^2 G_3}{\partial \Delta t^2}(S_n, I_n, R_n, \Delta t_R), \quad 0 < \Delta t_R < \Delta t.
\end{split}
\end{equation}
Similarly,
\begin{equation}\label{eq:38}
\begin{split}
&\Big|F_j(S(t_n), I(T_n), R(t_n)) - F_j(S_n, I_n, R_n)\Big|\\
&= \bigg|\dfrac{\partial F_j}{\partial S}(\zeta_j)(S(t_n) - S_n) + \dfrac{\partial F_j}{\partial I}(\zeta_j)(I(t_n) - I_n) + \dfrac{\partial F_j}{\partial R}(\zeta_j)(R(t_n) - R_n)\bigg|\\
&\leq \bigg|\dfrac{\partial F_j}{\partial S}(\zeta_j)\bigg||S(t_n) - S_n| + \bigg|\dfrac{\partial F_j}{\partial I}(\zeta_j)\bigg||I(t_n) - I_n| + \bigg|\dfrac{\partial F_j}{\partial R}(\zeta_j)\bigg||R(t_n) - R_n|\\
&\leq m_j^S|S(t_n) - S_n| + m_j^I|I(t_n) - I_n| + m_j^R|R(t_n) - R_n|,
\end{split}
\end{equation}
where $\zeta_j$ $(j = 1, 2, 3)$ are points between $(S(t_n), I(t_n), R(t_n))$ and $(S_n, I_n, R_n)$ and $m_j^S, m_j^I, m_j^R$ $(j = 1, 2, 3)$ are given by \eqref{eq:33}.\par
Let us denote
\begin{equation*}
e^S_n = S(t_n) - S_n, \quad e^I_n = I(t_n) - I_n, \quad e^R_n = R(t_n) - R_n, \quad e_n = |e_n^S| + |e_n^I| + |e_n^R|.
\end{equation*}
Then, from \eqref{eq:36}-\eqref{eq:38} we obtain
\begin{equation*}
\begin{split}
&|e_{n + 1}^S| \leq |e_n^S| + \Delta t \big(m_1^S|e_n^S| + m_1^I|e_n^I| + m_1^R|e_n^R|\big) + \dfrac{\Delta t^2}{2}(L_1 + M_1),\\
&|e_{n + 1}^I| \leq |e_n^I| + \Delta t \big(m_2^S|e_n^S| + m_2^I|e_n^I| + m_2^R|e_n^R|\big) + \dfrac{\Delta t^2}{2}(L_2 + M_2),\\
&|e_{n + 1}^R| \leq |e_n^R| + \Delta t \big(m_3^S|e_n^S| + m_3^I|e_n^I| + m_3^R|e_n^R|\big) + \dfrac{\Delta t^2}{2}(L_3 + M_3),
\end{split}
\end{equation*}
which implies that
\begin{equation}\label{eq:40}
e_{n + 1} \leq (1 + m\Delta t)e_n + \dfrac{\Delta t^2}{2}(L + M),
\end{equation}
where $L, m$ and $M$ are defined by \eqref{eq:32}, \eqref{eq:33} and \eqref{eq:35}, respectively.\par
For convenience, we set $\tau = L + M$. It follows from \eqref{eq:40} that
\begin{equation}\label{eq:41}
\begin{split}
e_{n + 1} &\leq (1 + m\Delta t)e_n + \tau\dfrac{\Delta t^2}{2}\\
&\leq (1 + m\Delta t)\bigg[(1 + m\Delta t)e_{n - 1} + \tau\dfrac{\Delta t^2}{2}\bigg] + \tau\dfrac{\Delta t^2}{2} = (1 + m\Delta t)^2e_{n - 1} + \tau\dfrac{\Delta t^2}{2}\Big[1 + (1 + m\Delta t)\Big]\\
&\leq \ldots \leq (1 + m\Delta t)^ne_0 + \tau\dfrac{\Delta t^2}{2}\sum_{j = 0}^n(1 + m\Delta t)^j = \tau\dfrac{\Delta t^2}{2}\dfrac{(1 + m\Delta t)^{n + 1} - 1}{(1 + m\Delta t) - 1} = \dfrac{\tau\Delta t}{2m}\big[(1 + m\Delta t)^{n + 1} - 1\big].
\end{split}
\end{equation}
Combining \eqref{eq:41} and the well-known inequality $e^x \geq x + 1$ for $x \geq 0$ we obtain
\begin{equation}\label{eq:42}
e_{n + 1} \leq \dfrac{\tau\Delta t}{2m}\big[e^{m(n + 1)t_n} - 1\big] = \dfrac{\tau\Delta t}{2m}\big[e^{mt_{n + 1}} - 1\big].
\end{equation}
This is the desired conclusion. The proof is complete.
\end{proof}
By the estimate \eqref{eq:42} we obtain the following result.
\begin{corollary}
The following estimate holds for the NSFD scheme \eqref{eq:23}
\begin{equation*}
|S(t_n) - S_n| + |I(t_n) - I_n| + |R(t_n) - R_n| \leq \dfrac{\tau\Delta t}{2m}\big[e^{mt_{n}} - 1\big]
\end{equation*}
for $n = 0, 1, 2, \ldots$.
\end{corollary}
\section{Numerical experiments}\label{sec4}
In this section, we report some numerical examples to support the theoretical results. For this purpose,  we consider the model \eqref{eq:1} with the incidence function
\begin{equation*}
f(S, I) = \dfrac{\alpha SI}{1 + aS + bI + cSI},
\end{equation*}
where $\alpha, a, b, c > 0$. More precisely, the model under consideration is
\begin{equation}\label{eq:45}
\begin{split}
&\dfrac{dS}{dt} = \Lambda - \dfrac{\alpha SI}{1 + aS + bI + cSI} - \big(\mu_0 + \nu\big)S,\\
&\dfrac{dI}{dt} = \dfrac{\alpha SI}{1 + aS + bI + cSI} - \big(\mu_0 + \mu_1 + \beta \big)I,\\
&\dfrac{dR}{dt} = \beta I + \nu S - \mu_0 R.
\end{split}
\end{equation}
The basic reproduction number of the model \eqref{eq:45} is given by
\begin{equation*}
\mathcal{R}_0 = \dfrac{\alpha\Lambda}{(\mu_0 + \mu_1 + \beta)(\mu_0 + \nu + a\Lambda)}.
\end{equation*}
\begin{example}[Dynamics of the NSFD scheme and standard ones]\label{example1}
Consider the model \eqref{eq:45} with following parameters. In Table \ref{table1}, the term "GAS" stands for the globally asymptotically stable equilibrium point. 
\begin{table}[H]
\begin{center}
\caption{The parameters used in Example \ref{example1}}\label{table1}
\begin{tabular}{ccccccccccccccc}
\hline
Parameter&Value&Source&Parameter&Value&Source&GAS\\
\hline
$\Lambda$&$0.8$&Assumed&$\mu_0$&$0.000232$&\cite{Martin}&$E^0 = (649.35,\, 0,\,   2798.90)$\\
$\alpha$&$0.5$&Assumed&$\nu$&$0.001$&Assumed\\
$a$&$0.8$&Assumed&$\mu_1$&$0.0000547$&\cite{MMWR}\\
$b$&$0.9$&Assumed&$\beta$&$0.8$&Assumed\\
$c$&$0.95$&Assumed&$\mathcal{R}_0$&$0.7795$&Computed\\
\hline
\end{tabular}
\end{center}
\end{table}
\end{example}
We use the standard Euler scheme, the second-order Runge-Kutta (RK2) scheme and the NSFD scheme \eqref{eq:23} with $\varphi(\Delta t) = \Delta t$ to solve the model with initial data $(S(0), I(0), R(0)) = (500, 100, 2000)$. The numerical solutions generated by these schemes with various step sizes are depicted in Figures \ref{fig:1}-\ref{fig:9}.\par
From the obtained results we see that the approximations  generated by the Euler and RK2 schemes are negative and unstable. Hence, the positivity, boundedness and stability of the HBV model are destroyed. The similar result can be found in previous works on NSFD schemes for differential equations \cite{Dang1, Dang2, Dang3, Dang4, Dang5, Dang6, Hoangnew, Hoang1}. Conversely, the approximations generated by the NSFD scheme preserve the essential qualitative properties of the continuous model regardless of the chosen step sizes. In particular, the results in Figures \ref{fig:7}-\ref{fig:9} indicate that the dynamics of the NSFD scheme does not depend on the step sizes. This can be explained by the theoretical analyses constructed in Section \ref{sec3}.
\newpage
\begin{figure}[H]
\centering
\includegraphics[height=10cm,width=18cm]{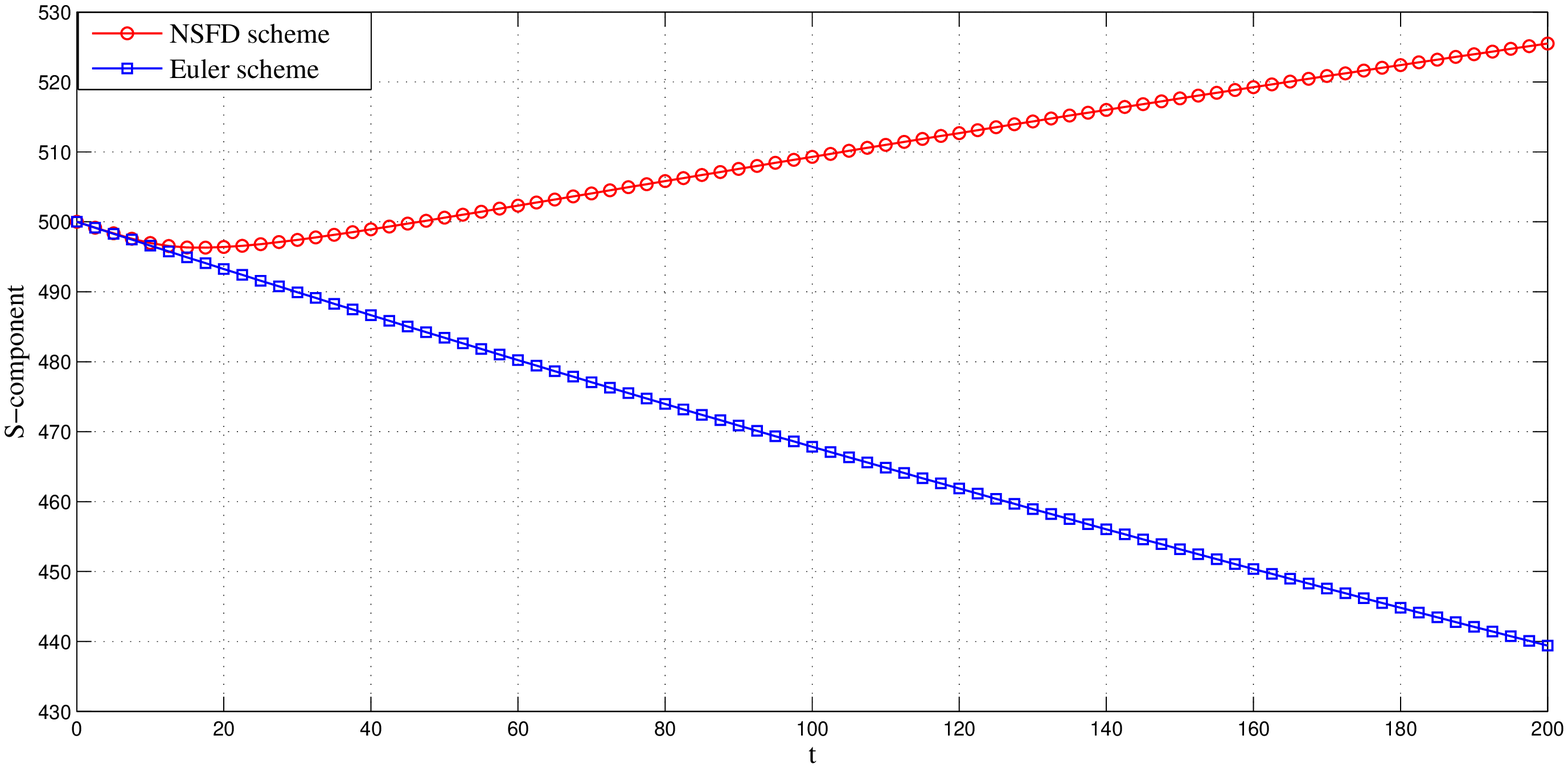}
\caption{The S-component generated by the Euler and NSFD schemes with $\Delta t = 2.5$ after $80$ iterations.}\label{fig:1}
\end{figure}
\begin{figure}[H]
\centering
\includegraphics[height=10cm,width=18cm]{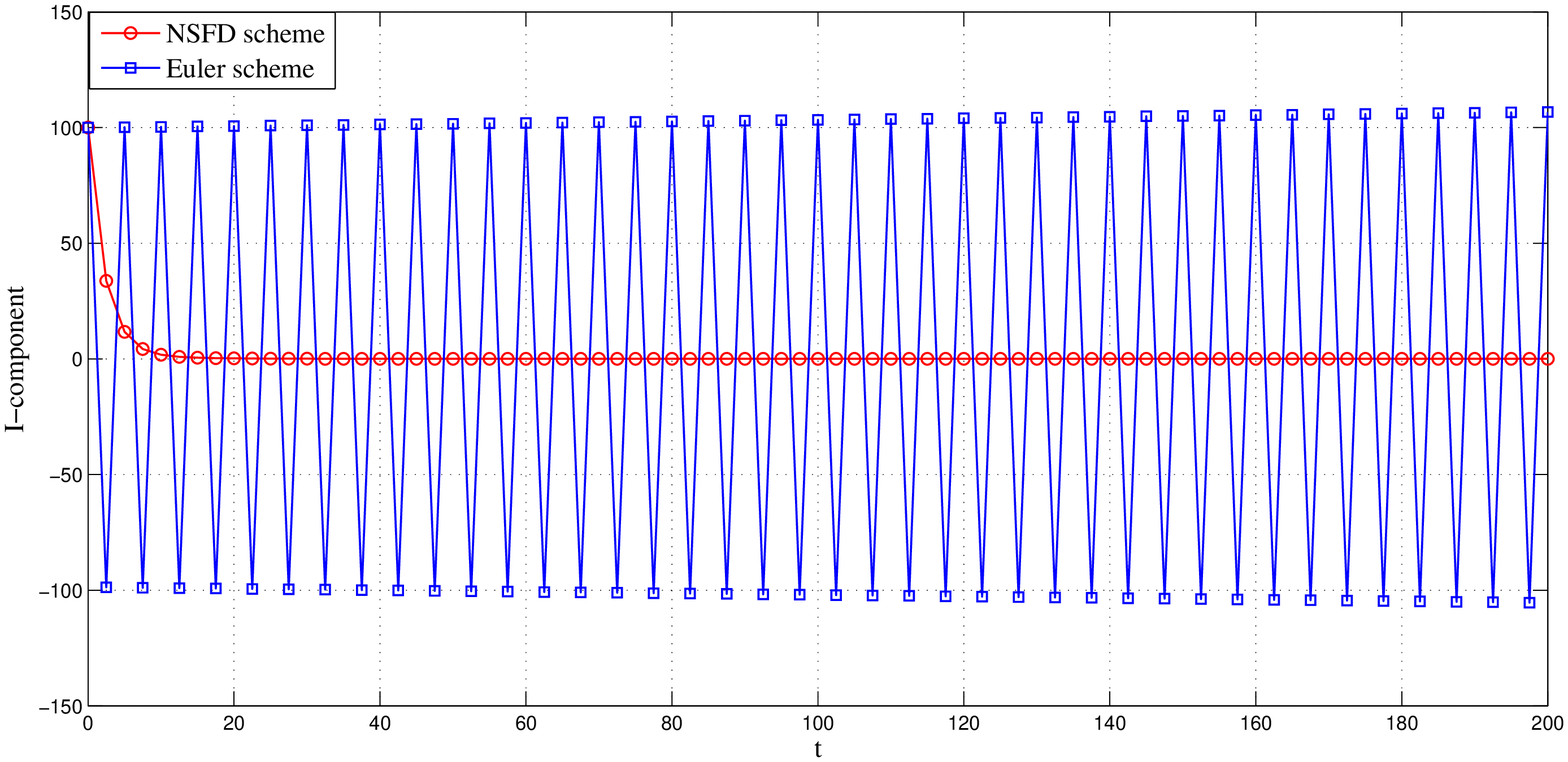}
\caption{The I-component generated by the Euler and NSFD schemes with $\Delta t = 2.5$ after $80$ iterations.}\label{fig:2}
\end{figure}
\begin{figure}[H]
\centering
\includegraphics[height=10cm,width=18cm]{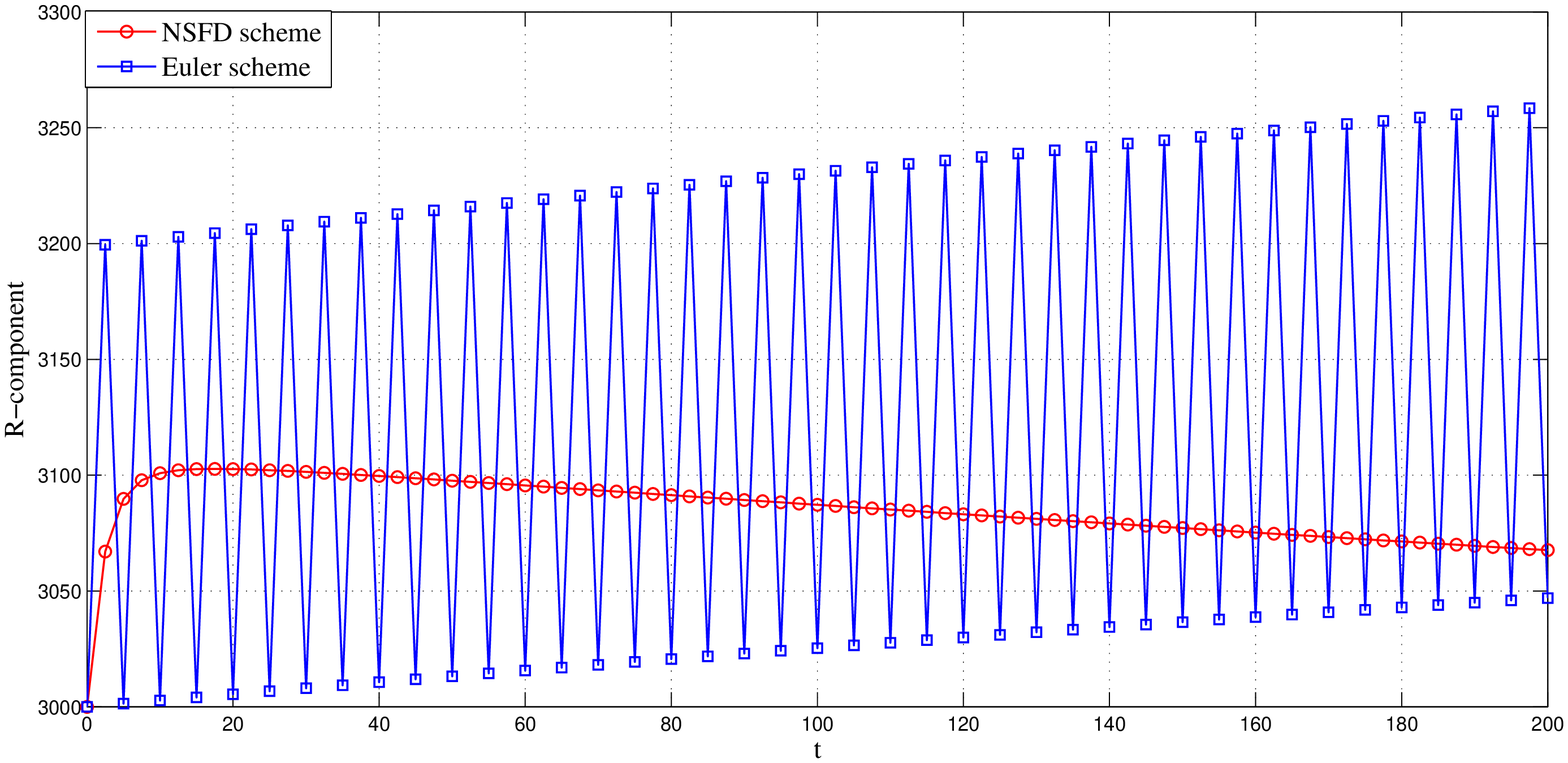}
\caption{The R-component generated by the Euler and NSFD schemes with $\Delta t = 2.5$ after $80$ iterations.}\label{fig:3}
\end{figure}
\begin{figure}[H]
\centering
\includegraphics[height=10cm,width=18cm]{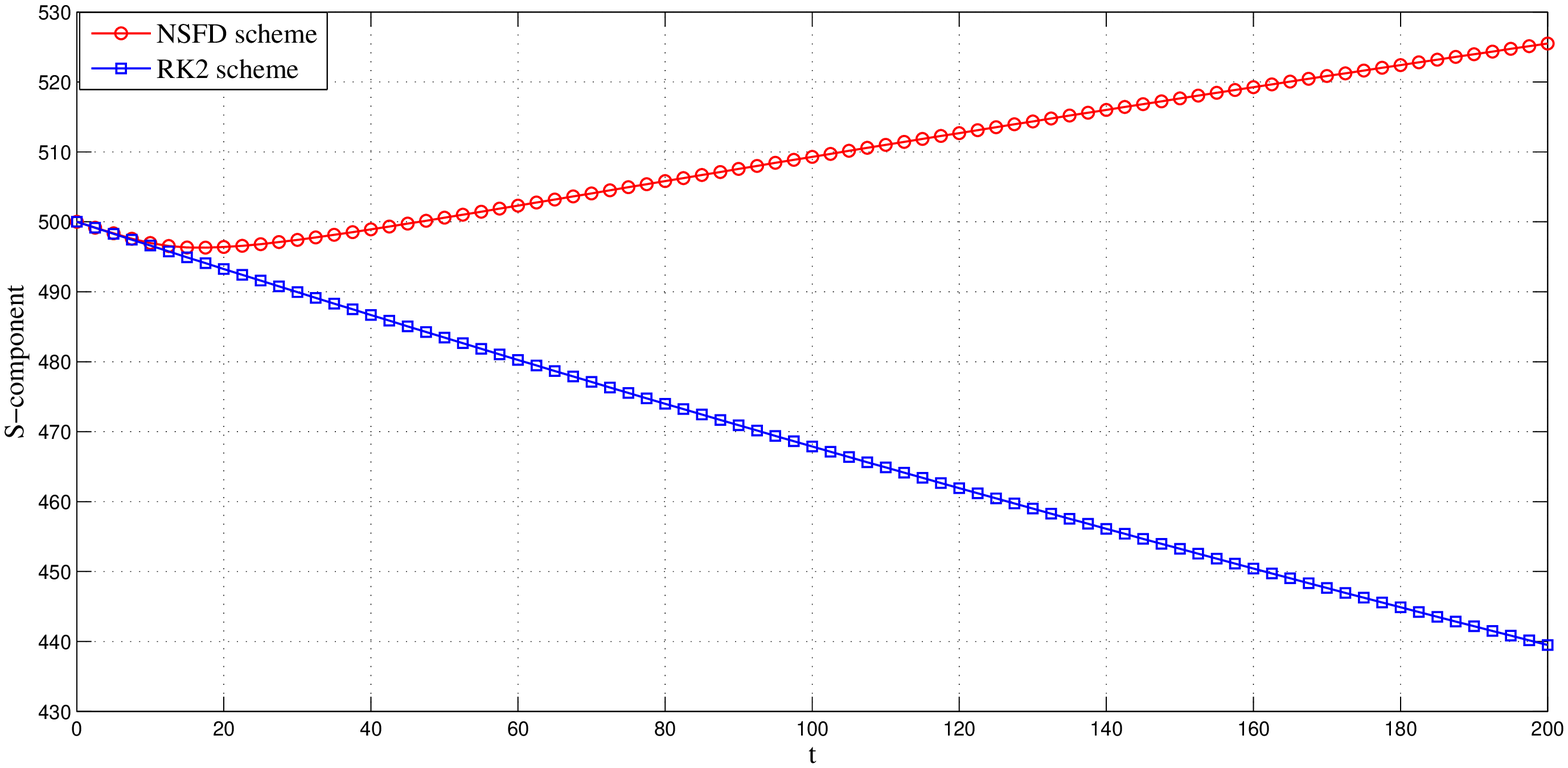}
\caption{The S-component generated by the RK2 and NSFD schemes with $\Delta t = 2.5$ after $80$ iterations.}\label{fig:4}
\end{figure}
\begin{figure}[H]
\centering
\includegraphics[height=10cm,width=18cm]{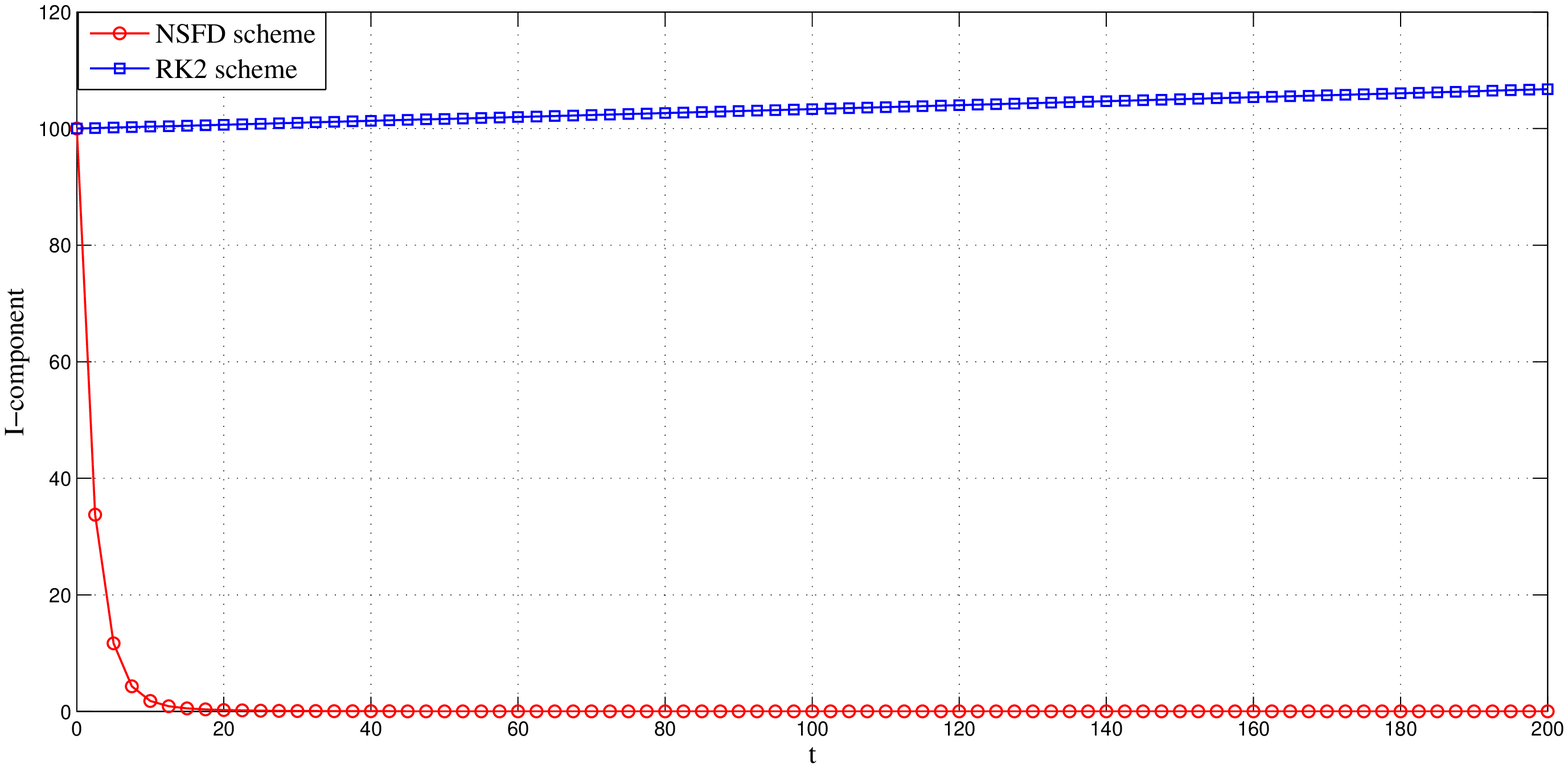}
\caption{The I-component generated by the RK2 and NSFD schemes with $\Delta t = 2.5$ after $80$ iterations.}\label{fig:5}
\end{figure}
\begin{figure}[H]
\centering
\includegraphics[height=10cm,width=18cm]{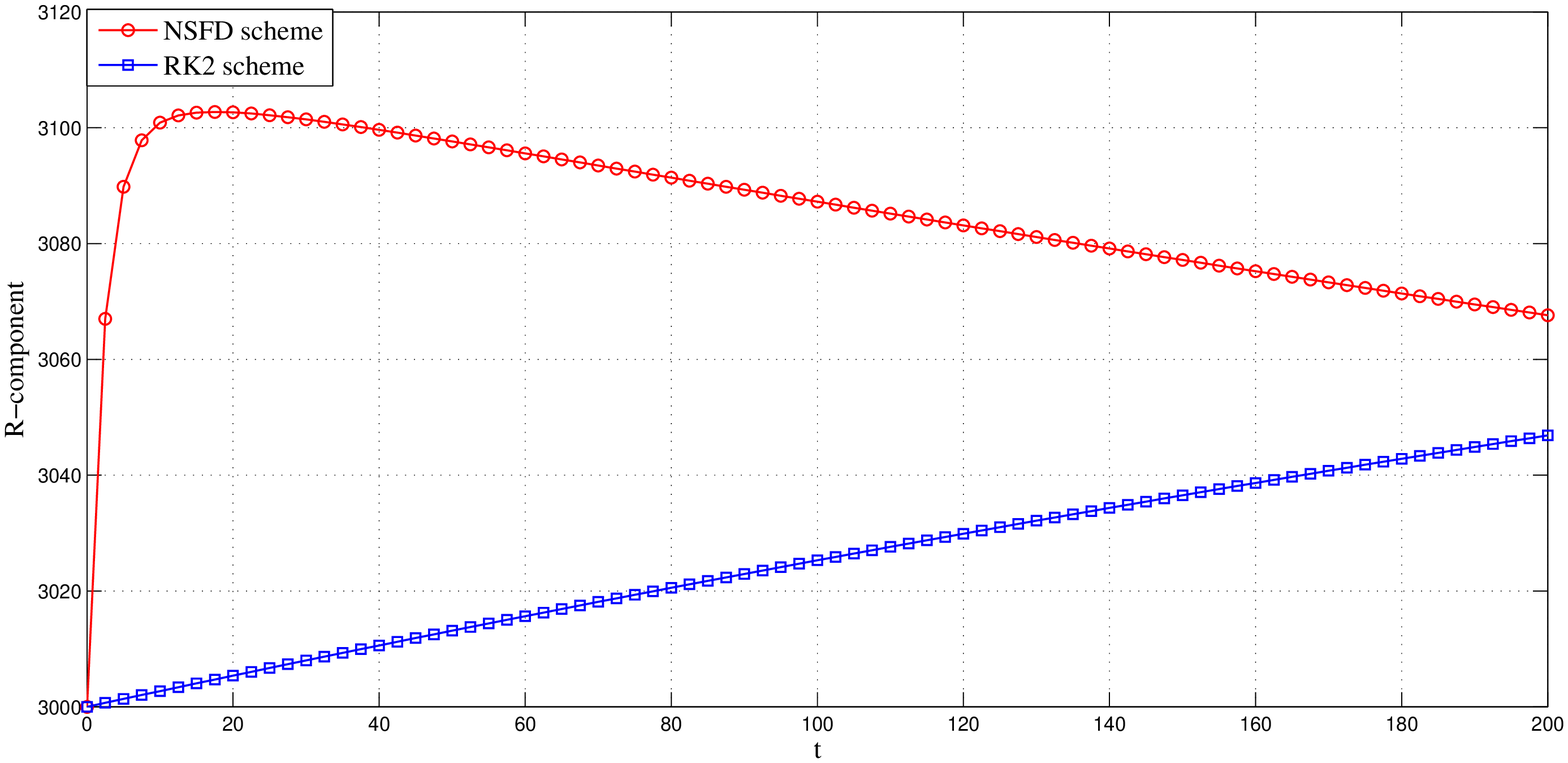}
\caption{The R-component generated by the RK2 and NSFD schemes with $\Delta t = 2.5$ after $80$ iterations.}\label{fig:6}
\end{figure}
\begin{figure}[H]
\centering
\includegraphics[height=10cm,width=18cm]{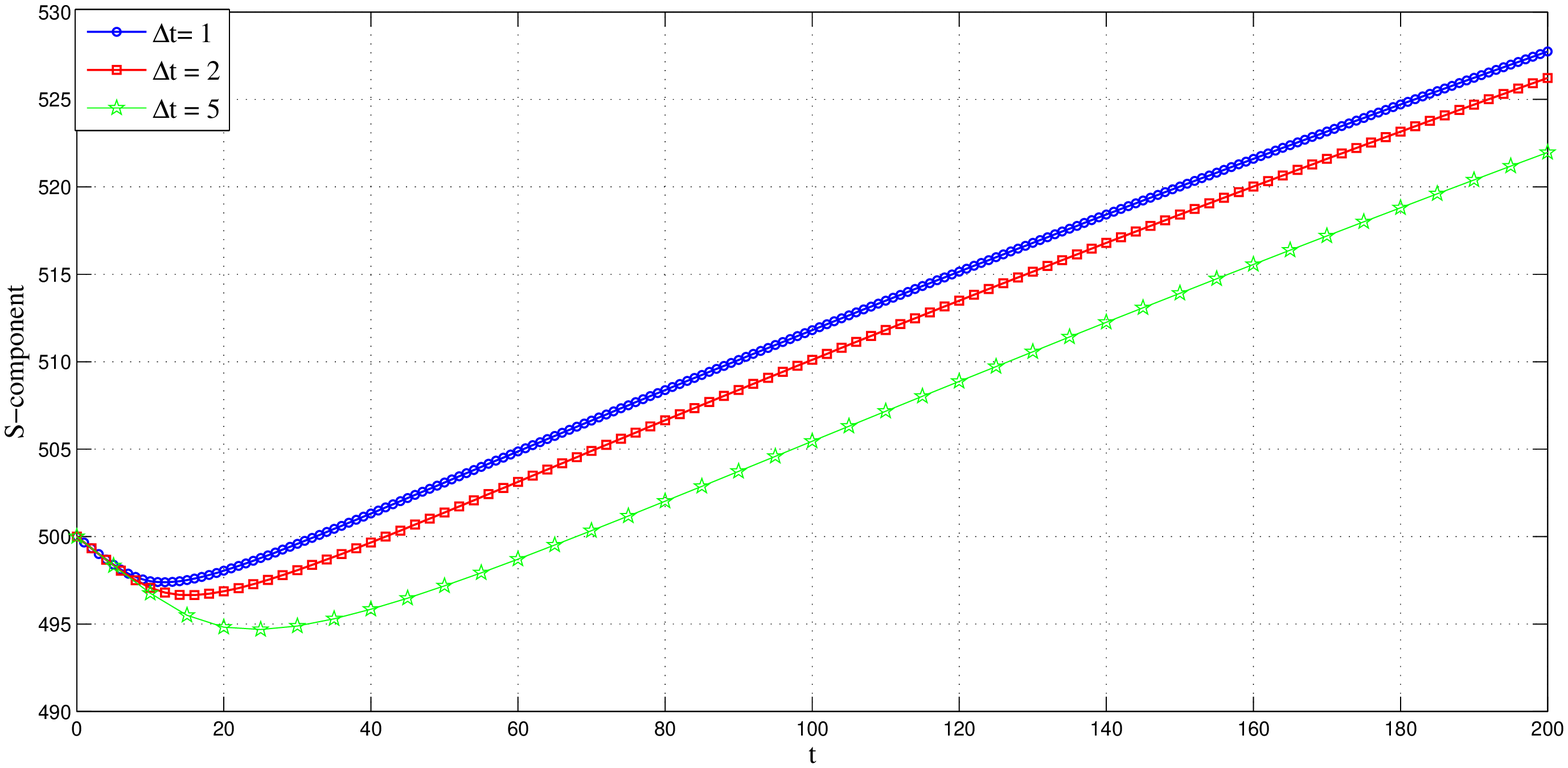}
\caption{The S-component generated by the NSFD scheme with $\Delta  t= 1.0$, $\Delta t = 2.0$ and $\Delta t = 5.0$.}\label{fig:7}
\end{figure}
\begin{figure}[H]
\centering
\includegraphics[height=10cm,width=18cm]{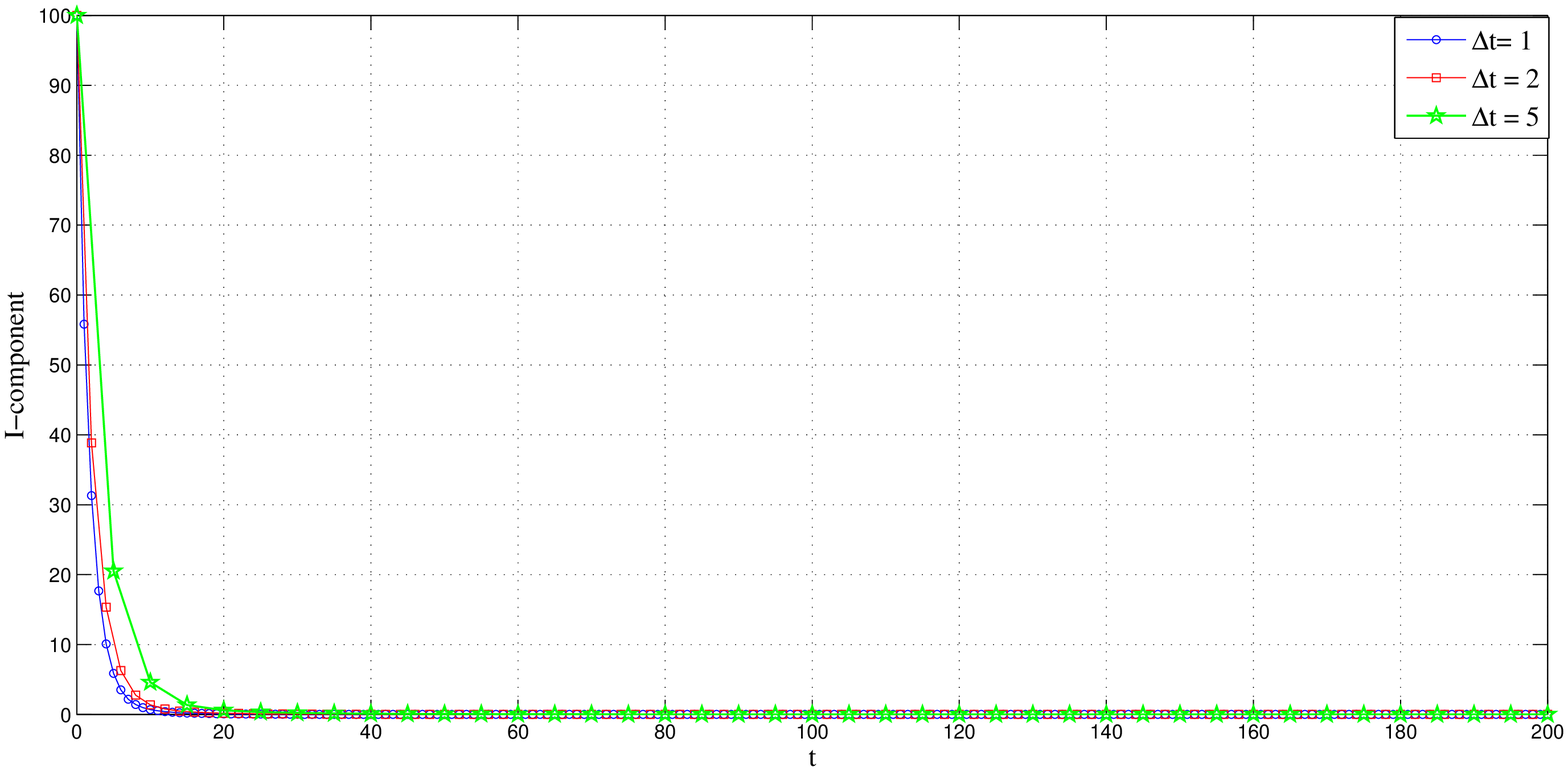}
\caption{The I-component generated by the NSFD scheme with $\Delta  t= 1.0$, $\Delta t = 2.0$ and $\Delta t = 5.0$.}\label{fig:8}
\end{figure}
\begin{figure}[H]
\centering
\includegraphics[height=10cm,width=18cm]{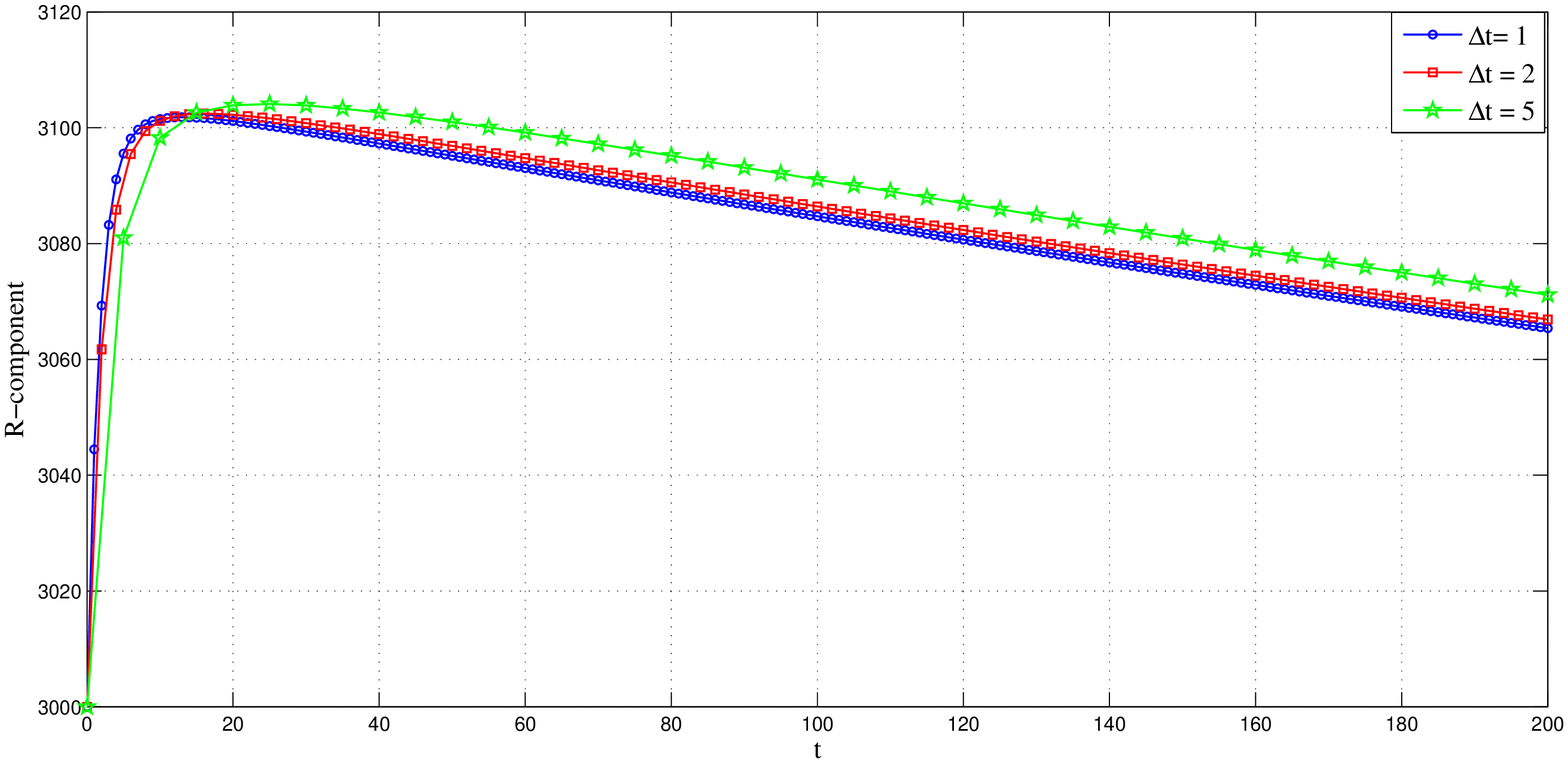}
\caption{The R-component generated by the NSFD scheme with $\Delta  t= 1.0$, $\Delta t = 2.0$ and $\Delta t = 5.0$.}\label{fig:9}
\end{figure}
\begin{example}[Dynamics of the HBV model when $\mathcal{R}_0 < 1$]\label{example2}
The aim of this example is to examine the dynamics of the model \eqref{eq:3} when $\mathcal{R}_0 < 1$. For this purpose, consider the model \eqref{eq:3} with following parameters.
\begin{table}[H]
\begin{center}
\caption{The parameters used in Example \ref{example2}}\label{table2}
\begin{tabular}{ccccccccccccccc}
\hline
Parameter&Value&Source&Parameter&Value&Source&GAS\\
\hline
$\Lambda$&$0.232$&\cite{Martin}&$\mu_0$&$0.000232$&\cite{Martin}&$E^0 = (126.64,\, 0,\,    873.36)$\\
$\alpha$&$0.0009$&\cite{Martin}&$\nu$&$0.0016$&Assumed\\
$a$&$0.25$&Assumed&$\mu_1$&$0.0000547$&\cite{MMWR}\\
$b$&$0.50$&Assumed&$\beta$&$0.25$&Assumed\\
$c$&$0.75$&Assumed&$\mathcal{R}_0$&$0.0139$&Computed\\
\hline
\end{tabular}
\end{center}
\end{table}
\end{example}
In this case, the DFE point $E^0$ is globally asymptotically stable. We use the NSFD scheme \eqref{eq:3} with $\varphi(\Delta t) = \Delta t$ and $\Delta t = 10^{-4}$ to simulate the dynamics of the HBV model. The obtained phase spaces are depicted in Figure \ref{fig:10}. In this figure, each blue curve represents a phase space corresponding to a specific initial data, the red circle indicates the location of the DFE point and the green arrows describe the evolution of the HBV model. It is clear that the GAS of $E^0$ as well as the dynamical properties of the HBV model are confirmed. 
\begin{figure}[H]
\centering
\includegraphics[height=12cm,width=15cm]{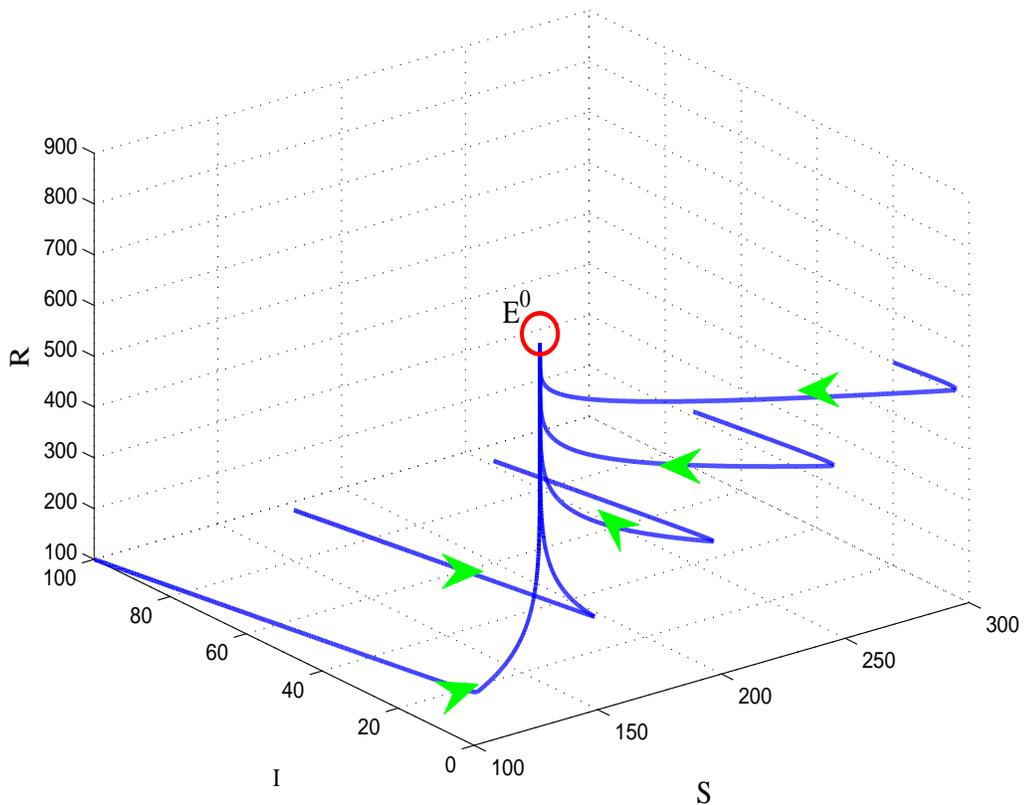}
\caption{The phase spaces of the HBV in Example \ref{example2}.}\label{fig:10}
\end{figure}

\begin{example}[Dynamics of the HBV model when $\mathcal{R}_0 > 1$]\label{example3}
In this example, we investigate the dynamics of the model \eqref{eq:3} when $\mathcal{R}_0 > 1$. For this reason, consider the model \eqref{eq:3} with following parameters.
\begin{table}[H]
\begin{center}
\caption{The parameters used in Example \ref{example3}}\label{table3}
\begin{tabular}{ccccccccccccccc}
\hline
Parameter&Value&Source&Parameter&Value&Source&GAS\\
\hline
$\Lambda$&$0.2$&Assumed&$\mu_0$&$0.000232$&\cite{Martin}&$E^* = (19.60, \,6.95,\,    1004.80)$\\
$\alpha$&$0.005$&Assumed&$\nu$&$0.0016$&Assumed\\
$a$&$0.5$&Assumed&$\mu_1$&$0.0000547$&\cite{MMWR}\\
$b$&$0.1$&Assumed&$\beta$&$0.025$&Assumed\\
$c$&$0.2$&Assumed&$\mathcal{R}_0$&$0.0139$&Computed\\
\hline
\end{tabular}
\end{center}
\end{table}
\end{example}
In this case, the DEE point $E^*$ is globally asymptotically stable. We use the NSFD scheme \eqref{eq:3} with $\varphi(\Delta t) = t$ and $t = 10^{-4}$ to solve the HBV model. The obtained phase spaces are depicted in Figure \ref{fig:11}. Similarly to Example \ref{example2}, the GAS of $E^0$ as well as the dynamical properties of the model are shown.
\begin{figure}[H]
\centering
\includegraphics[height=12cm,width=15cm]{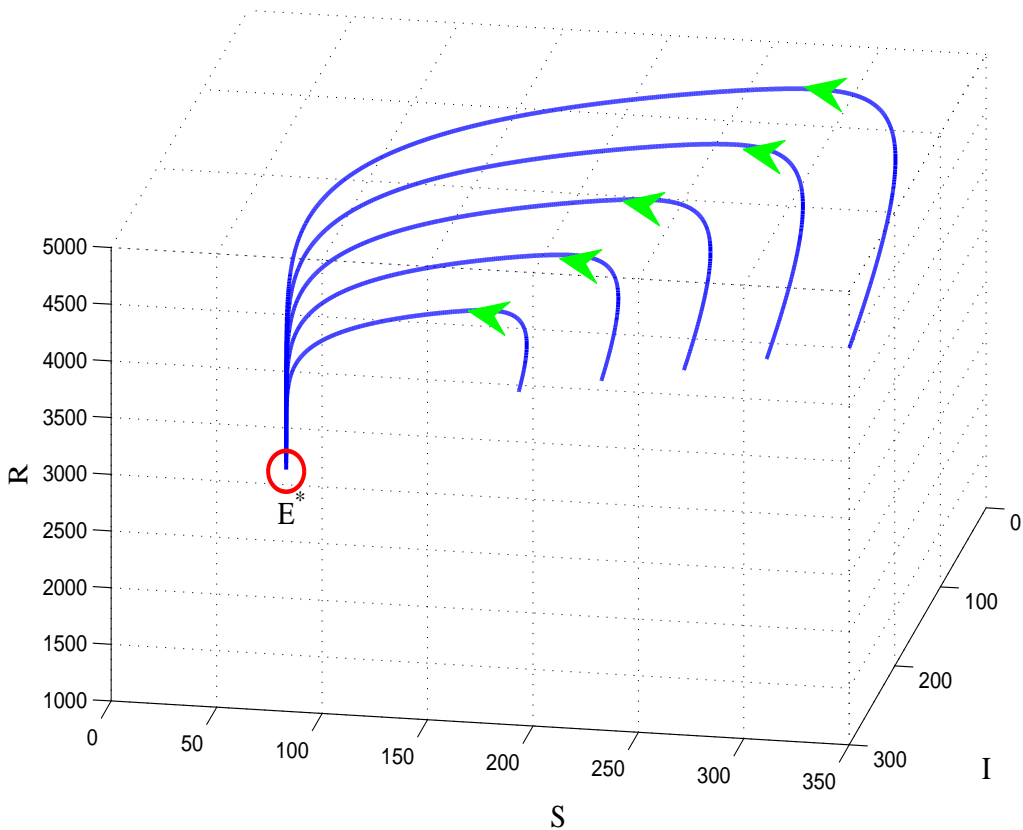}
\caption{The phase spaces of the HBV in Example \ref{example3}.}\label{fig:11}
\end{figure}
\section{Conclusions and open problems}\label{sec5}
In this work, we have studied dynamics of a generalized hepatitis B epidemic model and its dynamically consistent NSFD model. The positivity, boundedness, the basic reproduction number and asymptotic stability properties of the model have been analyzed rigorously. It was proved, by the Lyapunov stability theory and the Poincare-Bendixson theorem in combination with the Bendixson-Dulac criterion, that the DFE point is globally asymptotically stable if the basic reproduction number $\mathcal{R}_0 \leq 1$ and the DEE point is globally asymptotically stable whenever $\mathcal{R}_0 > 1$. Besides, the Mickens’ methodology was applied to formulate a dynamically consistent NSFD model for the continuous model. By rigorously mathematical analyses we have shown that the constructed NSFD scheme preserves essential mathematical features of the continuous model for all finite step sizes. Finally, the numerical experiments are conducted to illustrate the theoretical findings and to demonstrate advantages of the NSFD scheme over standard ones. The obtained results provided important improvements for the ones presented in \cite{Khan} and \cite{Suryanto}.\par
Based on the approach in \cite{Khan}, we can also use the optimal control strategy for the model \eqref{eq:3} to eliminate the spreading of the HBV. This is a very interesting and important problem with many useful applications in real-world applications.\par
Although the constructed NSFD scheme is only convergent of order $1$, its main advantage is that it can preserve the essential properties of the continuous model for all finite step sizes. On the other hand, it can be easily combined with extrapolation techniques or variable step-size strategies to improve the accuracy. The most important thing is that it can be operate well regardless of step sizes.\par
In the near future, optimal control strategies and high-order NSFD schemes for the model \eqref{eq:3} and its extensions will be studied. Additionally, the practice applications of the HBV model \eqref{eq:3}  will be also considered.


\section*{References}
\begin{thebibliography}{00}
\bibitem{Adekanye}
O. Adekanye, T. Washington, Nonstandard finite difference scheme for a Tacoma Narrows Bridge model, Applied Mathematical Modelling 62(2018) 223-236.
\bibitem{Ahmad}
S. Ahmad, M. Rahman, M. Arfan, On the analysis of semi-analytical solutions of Hepatitis B epidemic model under the Caputo-Fabrizio operator, Chaos, Solitons \& Fractals 146 (2021) 110892.
\bibitem{Allen}
L. J. S. Allen, An introduction to mathematical biology, Pearson, 2006.
\bibitem{Allen1}
L. J. S. Allen, P. van den Driessche, The basic reproduction number in some discrete-time epidemic models, Journal of Difference Equations and Applications 14(2008) 1127-1147.
\bibitem{Arenas}
A. J. Arenas, G. Gonzalez-Parra, B. M. Chen-Charpentier, Construction of nonstandard finite difference schemes for the SI and SIR epidemic models of fractional order, Mathematics and Computers in Simulation 121 (2016) 48-63. 
\bibitem{Ascher}
U. M. Ascher, L. R. Petzold, Computer Methods for Ordinary Differential Equations and Differential-Algebraic Equations,  Society for Industrial and Applied Mathematics, Philadelphia, 1998.
\bibitem{Bonhoeffer}
S. Bonhoeffer, R. M. May, G. M. Shaw, M. A. Nowak, Virus dynamics and drug therapy. Proc. Natl Acad. Sci. USA 94(1997) 6971-6976.
\bibitem{Brauer}
F. Brauer, The Kermack-McKendrick epidemic model revisited, Mathematical Biosciences 198(2005)119-131.
\bibitem{Brauer1}
F. Brauer, Mathematical epidemiology: Past, present, and future,  Infectious Disease Modelling  2(2017)113-127.
\bibitem{Calatayud}
J. Calatayud, M. Jornet, An improvement of two nonstandard finite difference schemes for two population mathematical models, Journal of Difference Equations and Applications 27(2021)422-430.
\bibitem{Cardoso}
L. C. Cardoso, R. F. Camargo, F. L. P. dos Santos, J. P. C. D. Santos, Global stability analysis of a fractional differential system in hepatitis B, Chaos, Solitons \& Fractals 143(2021) 110619.
\bibitem{Cardoso1}
L. C. Cardoso, F. L. P. Dos Santos, R. F. Camargo, Analysis of fractional-order models for hepatitis B, Computational and Applied Mathematics 37(2018) 457-4586.
\bibitem{Cresson}
J. Cresson, F. Pierret, {Non standard finite difference scheme preserving dynamical properties},  Journal of Computational and Applied Mathematics 303(2016) 15-30.
\bibitem{Danane}
J. Danane, K. Allali, Z. Hammouch, Mathematical analysis of a fractional differential model of HBV infection with antibody immune response, Chaos, Solitons \& Fractals 136(2020) 109787.
\bibitem{Dang1}
Quang A Dang, Manh Tuan Hoang, Positivity and global stability preserving NSFD schemes for a mixing propagation model of computer viruses, Journal of Computational and Applied Mathematics
374(2020) 112753.
\bibitem{Dang2}
Quang A Dang, Manh Tuan Hoang, Nonstandard finite difference schemes for a general predator-prey system, Journal of Computational Science 36(2019) 101015.
\bibitem{Dang3}
Quang A Dang, Manh Tuan Hoang, Dynamically consistent discrete metapopulation model, Journal of Difference Equations and Applications 22 (2016) 1325-1349.
\bibitem{Dang4}
Quang A Dang, Manh Tuan Hoang, Lyapunov direct method for investigating stability of nonstandard finite difference schemes for metapopulation models, Journal of Difference Equations and Applications 24(2019) 15-47.
\bibitem{Dang5}
Quang A Dang, Manh Tuan Hoang, Complete global stability of a metapopulation model and its dynamically consistent discrete models, Qualitative theory of dynamical systems 18 (2019) 461-475.
\bibitem{Dang6}
Quang A Dang, Manh Tuan Hoang, Positive and elementary stable explicit nonstandard Runge-Kutta methods for a class of autonomous dynamical systems, International Journal of Computer Mathematics 97 (2020) 2036-2054.
\bibitem{Din}
A. Din, Y. Li, Q. Liu, Viral dynamics and control of hepatitis B virus (HBV) using an epidemic model, Alexandria Engineering Journal 59(2020) 667-679.
\bibitem{Din1}
A. Din, Y. Li, A. Yusuf, Delayed hepatitis B epidemic model with stochastic analysis, Chaos, Solitons \& Fractals 146 (2021) 110839.
%
\bibitem{Din2}
W. D. Qin, Q. Ma, Z. Y. Man, X. H. Ding, A boundedness and monotonicity preserving method for a generalized population model, Journal of Difference Equations and Applications  26(2020)1347-1368.
\bibitem{Gao}
F. Gao, X. Li, W. Li, X. Zhou, Stability analysis of a fractional-order novel hepatitis B virus model with immune delay based on Caputo-Fabrizio derivative, Chaos, Solitons \& Fractals 142(2021) 110436.
\bibitem{Garba}
S. M. Garba, A. B. Gumel, A. S. Hassan, J. M. -S. Lubuma, Switching from exact scheme to nonstandard finite difference scheme for linear delay differential equation, Applied Mathematics and Computation
258(2015) 388-403.
\bibitem{Gupta}
M. Gupta, J. M. Slezak, F. Alalhareth, S. Roy, H. V. Kojouharov, Second-order Nonstandard Explicit Euler Method, AIP Conference Proceedings 2302(2020) 110003.
\bibitem{Grassly}
N. C. Grassly, C. Fraser, Mathematical models of infectious disease transmission, Nature Reviews Microbiology 6(2008) 477-487.
\bibitem{Hethcote}
H. W. Hethcote, The mathematics of infectious diseases. SIAM Rev. 42((2000) 599-653.
\bibitem{Hoang}
M. T. Hoang, O. F.  Egbelowo O.F, Dynamics of a Fractional-Order Hepatitis B Epidemic Model and Its Solutions by Nonstandard Numerical Schemes. In: Hattaf K., Dutta H. (eds) Mathematical Modelling and Analysis of Infectious Diseases. Studies in Systems, Decision and Control, vol 302. Springer, Cham. https://doi.org/10.1007/978-3-030-49896-2\_5.
\bibitem{Hoang1}
M. T. Hoang, O.F. Egbelowo, On the global asymptotic stability of a hepatitis B epidemic model and its solutions by nonstandard numerical schemes, Boletín de la Sociedad Matemática Mexicana 26(2020) 1113-1134.
\bibitem{Hoangnew}
Manh Tuan Hoang, Reliable approximations for a hepatitis B virus model by nonstandard numerical schemes, Mathematics and Computers in Simulation 193(2022) 32-56.
\bibitem{Hoang1}
Manh Tuan Hoang, A. M. Nagy, Uniform asymptotic stability of a Logistic model with feedback control of fractional order and nonstandard finite difference schemes, Chaos, Solitons \& Fractals 123(2019) 24-34.
\bibitem{Horvath}
Z. Horv\'ath, Positivity of Runge-Kutta and diagonally split Runge-Kutta methods, Applied Numerical Mathematics 28(1998) 309-326.

\bibitem{Karaji}
P. T. Karaji, N. Nyamoradi, Analysis of a fractional SIR model with General incidence function, Applied Mathematics Letters 108(2020) 106499.
\bibitem{Kermack}
W. O. Kermack, A. G. McKendrick, A Contribution to the Mathematical Theory of Epidemics, Proceedings of the Royal Society of London, Series A 115(1927) 700-721.
\bibitem{Khan}
T. Khan, Z. Ullah, N. Ali, G. Zaman, Modeling and control of the hepatitis B virus spreading using an epidemic model, Chaos, Solitons and Fractals 124 (2019) 1-9.
\bibitem{Khan1}
A. Khan, G. Hussain,  M. Inc,  G. Zaman, Existence, uniqueness, and stability of fractional hepatitis B epidemic model, Chaos 30(2020) 103104.
\bibitem{Khan2}
T. Khan, A. Khan, G. Zaman, The extinction and persistence of the stochastic hepatitis B epidemic model, Chaos, Solitons \& Fractals 108(2018)  123-128.
\bibitem{Kojouharov}
H. V. Kojouharov, S. Roy, M. Gupta, F. Alalhareth, J. M. Slezak, A second-order modified nonstandard theta method for one-dimensional autonomous differential equations, Applied Mathematics Letters
112(2021) 106775.
\bibitem{LaSalle}
J. La Salle, S. Lefschetz, Stability by Liapunov’s Direct Method, Academic Press, New York, 1961.
%
\bibitem{Li}
M. Y. Li, An Introduction to Mathematical Modeling of Infectious Diseases, Springer International Publishing AG,  2018.
\bibitem{Lyapunov}
A. M. Lyapunov, The general problem of the stability of motion, International Journal of Control,
Taylor \& Francis, 1992.
\bibitem{Manna}
K. Manna, S. P. Chakrabarty, Global stability of one and two discrete delay models for chronic hepatitis B infection with HBV DNA-containing capsids, Computational and Applied Mathematics 36(2017) 525-536.
\bibitem{Martcheva}
M. Martcheva, An Introduction to Mathematical Epidemiology, Springer, New York, 2015.
\bibitem{Martin}
N. K. Martin, P. Vickerman, M. Hickman, Mathematical modelling of hepatitis C treatment for injecting drug users, Journal of Theoretical Biology 274(2011) 58-66.
\bibitem{McCluskey}
C. C. McCluskeya, Y. Yang, Global stability of a diffusive virus dynamics model with general
incidence function and time delay, Nonlinear Analysis: Real World Applications 25(2015) 64-78.
\bibitem{McNabb}
A. McNabb, Comparison theorems for differential equations, Journal of Mathematical Analysis and Applications 119(1986) 417-428.
\bibitem{Mickens1}
R. E. Mickens, Nonstandard Finite Difference Models of Differential Equations, World Scientific, 1993.
\bibitem{Mickens2}
R. E. Mickens, Applications of Nonstandard Finite Difference Schemes, World Scientific, 2000.
\bibitem{Mickens3}
R. E. Mickens, Advances in the Applications of Nonstandard Finite Difference Schemes, World Scientific, 2005.
\bibitem{Mickens4}
R. E. Mickens, Nonstandard Finite Difference Schemes for Differential Equations, Journal of Difference Equations and Applications 8(2002) 823-847.
\bibitem{Mickens5}
R. E. Mickens, Nonstandard Finite Difference Schemes: Methodology and Applications,  World Scientific, 2020.
\bibitem{Mickens6}
R. E. Mickens, I. H. Herron, Approximate rational solutions to the Thomas-Fermi equation based on dynamic consistency, Applied Mathematics Letters 116(2021)106994.
\bibitem{MMWR}
MMWR, Progress in hepatitis B prevention through universal infant vaccination China, 1997-2006, Morbidity and mortality weekly report 2007(18) 441-5 . 
\bibitem{Nowak}
M. A. Nowak, R. M. May, Virus Dynamics: Mathematical Principles of Immunology and Virology, New York: Oxford University Press 2000.
\bibitem{Patidar1}
K. C. Patidar, On the use of nonstandard finite difference methods, Journal of Difference Equations and Applications 11(2005) 735-758.
\bibitem{Patidar2}
K. C. Patidar, Nonstandard finite difference methods: recent trends and further developments, Journal of Difference Equations and Applications 22(2016) 817-849.
\bibitem{Shah}
S. A. A. Shah, M. A. Khan, M. Farooq, S. Ullah, E. O.Alzahrani, A fractional order model for Hepatitis B virus with treatment via Atangana-Baleanu derivative, Physica A: Statistical Mechanics and its Applications 538 (2020) 122636.
\bibitem{Suryanto}
A. Suryanto, I. Darti, On the nonstandard numerical discretization of SIR epidemic model with a
saturated incidence rate and vaccination, AIMS Mathematics 6(2020) 141-155.
%
\bibitem{Tian}
Y. Tian, X. Liu, Global dynamics of a virus dynamical model with general incidence rate and cure rate, Nonlinear Analysis: Real World Applications 16(2014) 17-26.
\bibitem{vdDriessche}
P. van den Driessche, J. Watmough, Reproduction numbers and sub-threshold endemic equilibria for compartmental models of disease transmission, Mathematical Biosciences 180(2002) 29-48.

%
\bibitem{Wood1}
D. T. Wood, H. V. Kojouharov, A class of nonstandard numerical methods for autonomous dynamical systems, Applied Mathematics Letters 50(2015) 78-82.
\bibitem{Wood2}
D. T. Wood, H. V. Kojouharov, D. T. Dimitrov, Universal approaches to approximate biological systems with nonstandard finite difference methods, Mathematics and Computers in Simulation 133(2017) 337-350.
\end{thebibliography}
\end{document}